\documentclass[11pt]{amsart}
\usepackage{bm}
\usepackage{bbm}
\usepackage{wasysym}
\usepackage{cite}
\usepackage{mathrsfs}
\usepackage{cases}
\usepackage{latexsym}
\usepackage[all]{xy}
\usepackage{stmaryrd}
\usepackage{amsfonts,enumitem}
\usepackage{float}
\usepackage{amssymb,amscd,amsthm,dsfont,pb-diagram}
\usepackage{amsmath}
\usepackage{fancyhdr}
\usepackage{amsxtra,ifthen}
\usepackage{verbatim}
\usepackage{tikz-cd}
\usepackage[vcentermath]{youngtab}
\usepackage{xcolor}
\usepackage[hidelinks]{hyperref}
\usepackage{young}
\usepackage{url}
\usepackage{indentfirst}
\usepackage[table]{xcolor}
\usepackage{ytableau} 
\usepackage{tikz} 
\usepackage{multirow,array} 
\usepackage{ragged2e}
\usepackage[utf8]{inputenc} 
\usepackage[dvipsnames]{xcolor} 
\usepackage{subcaption}
\usepackage{float}
\usepackage{dynkin-diagrams}
\usepackage{geometry}
\usetikzlibrary{arrows,decorations.markings}
\tikzstyle{dot}=[draw, fill =black, circle, inner sep=0pt, minimum size=2pt]

\tikzset{
	mycell/.style={
		rectangle, 
		minimum size=1.5cm, 
		inner sep=0pt, 
		font=\Large,
		draw, 
		anchor=south west
	},
	bluecell/.style={mycell, fill=blue!20},
	graycell/.style={mycell, fill=gray!20},
	emptycell/.style={mycell}
}

\theoremstyle{plain}

\numberwithin{equation}{subsection}
\newtheorem{theorem}{Theorem}[section]
\newtheorem{lemma}[theorem]{Lemma}
\newtheorem{proposition}[theorem]{Proposition}
\newtheorem{corollary}[theorem]{Corollary}
\newtheorem{conjecture}[theorem]{Conjecture}

\newtheorem{definition}[theorem]{Definition}
\newtheorem{example}[theorem]{Example}

\newtheorem{remark}[theorem]{Remark}
\newtheorem{remarks}[theorem]{Remarks}

\newcommand{\lam}{\lambda}

\newcommand{\B}{\mathsf{B}}

\def\bs{{\mathbf s}}

\def\blam{{\boldsymbol \lambda}}
\def\bmu{{\boldsymbol \mu}}

\def\bs{{\bf s}}
\def\bu{{\bf u}}

\def\<{\langle}
\def\>{\rangle}

\begin{document}

\title[Core abaci and Diophantine equations I]
{Core abaci and Diophantine equations I: Fundamental weight}

\thanks{Corresponding Author: Shasha Zhu}

\author{Yanbo Li}

\address{Li: School of Mathematics and Statistics, Northeastern
University at Qinhuangdao, Qinhuangdao, 066004, P.R. China}

\email{liyanbo707@163.com}

\author{Jiansheng Zhang}

\address{Zhang: Department of mathematics, School of Science, Northeastern
University, Shenyang, 110819, P.R. China}

\email{786579992@qq.com}

\author{Shasha Zhu}

\address{Zhu: Department of mathematics, School of Science, Northeastern
University, Shenyang, 110819, P.R. China}

\email{zhushasha202510@163.com}

\begin{abstract}
In the light of a series of papers on moving vectors, we define and study core abaci of classical affine types for arbitrary charge.
This greatly extends the concept of cores with charge zero, and make us being able to parameterize the affine Grassmannian $W^j$ 
by core abaci of charge $j$ for arbitrary classical affine types.
By associating a core abacus $(\lam, j)$ to a weight $\Lambda_j-\beta$ and an affine Weyl group element $w_{\lam, j}$, we prove that the height of $\beta$ 
is equal to the atomic length of $w_{\lam, j}$.
This solves a generalized version of the open problem raised by Brunat, Chapelier-Laget and Gerber.
Moreover, Diophantine equations of classical affine types
are established by using the height formula that given by Uglov vector. 
The solutions of certain classes of these Diophantine equations are proved to be completely parameterised by core abaci. 
As another application, closed formulae for computing the number of certain kinds of core abaci are given.
\end{abstract}




\maketitle


\tableofcontents

\section{Introduction}

Core partitions play an important role in different fields of mathematics
such as representation theory, enumerative combinatorics, number theory and so on.
We refer the reader to \cite{A, BDF, BCG, F2, GKS, GO, JL} for some details. In \cite{HJ2}, Hanusa and Jones generalized
the notion to types $B$, $C$ and $D$, which are all described by cores of type $A$.

For a core partition $\lam$, the size (number of nodes in the Young diagram $[\lam]$) is undoubtedly essential. 
In \cite{TW} Thiel and Williams established a formula for the size of 
type $A$ cores by using coroot lattice, and the formula was generalized to other classical affine types in \cite{STW}.
In 2024 Chapelier-Laget and Gerber introduced the so-called atomic length that is defined on the (affine) Weyl
group. This recover the constructions in \cite{TW,STW}. In particular, by regarding a core $\lam$ as an affine Grassmannian element 
at the fundamental weight $\Lambda_0$, it was shown that the size of $\lam$
is equal to the atomic length of a representative with minimal length in the affine Grassmannian. 
The quadratic formula for the atomic length was then used to establish certain Diophantine equations 
by Brunat, Chapelier-Laget and Gerber in \cite{BCG}. As a result, the solutions of Diophantine equations 
established could be partially parameterized by the corresponding cores. 
Moreover, using the atomic length associated with the fundamental weight $\Lambda_1$
in type $C_2^{(1)}$, they obtain a description of the solutions of the Diophantine equation
$x^2+y^2= 8N+1$ in terms of extended Grassmannian elements. Inspired by this result, a natural problem that has been proposed
by Brunat, Chapelier-Laget and Gerber is to find a combinatorial model for arbitrary twisted or untwisted affine types in which the size of cores
is given by the $\Lambda_1$ atomic length. This is precisely the starting point of our research.
The key of our study is to extend the notion of cores to arbitrary fundamental weight $\Lambda_i$.

\smallskip

By fixing an integer $j$, called charge, a partition $\lam$ can be expressed by an abacus, which first
appeared in the work of James \cite{J}. Given a pair $(\lam, j)$, one can associate it with a weight $\Lambda-\beta$. 
It is well known that if $\lam$ is a core, then the defect of $\Lambda-\beta$ is zero. In our opinion, the property of 
being defect zero is the essence of cores (of type $A^{(1)}$). We extend the definition of cores to arbitrary classical types of arbitrary charges 
based on this view. It is helpful to point out that so far all the definitions of cores are of charge 0 as far as we know.
Note that each core abacus $(\lam, j)$ can be obtained from the empty abacus 
by the action of a minimal length element $w_{\lam, j}$ in the corresponding affine Weyl group, and consequently
we can parameterize the affine Grassmannian by the collection of $j$-core abaci for arbitrary classical affine types, where $j$ is a charge.
This is the first gain of our generalization of cores.

Thanks to a series of papers \cite{LQ, LQT, HLQ} on the so-called moving vectors, which were introduced 
by Li and Qi \cite{LQ} to study the representation type of blocks of an Ariki-Koike algebra and were extended to
all twisted and untwisted classical affine types by the first author and his collaborators,
we can describe cores by abacus language. This makes it possible to give a formula on the height of a core by the so-called
Uglov vector obtained from Uglov map. To obtain this formula, we prove that the $\mathrm{q}$-vector obtained from 
the semidirect product decomposition of $w_{\lam, j}$ is equal to the sum of an adjustment item and the weighted Uglov vector obtained from general Uglov map.   
Based on the height formula, we find that for a core abacus $(\lam, j)$ the atomic length of $w_{\lam, j}$ is 
equal to the height of $\beta$, which is just the size of $\lam$ in some cases.
The result can be viewed as an answer of the open problem raised by Brunat, Chapelier-Laget and Gerber in \cite{BCG}. 
This is the second gain of our generalization of cores.
The following diagram illustrates the interrelationships among these results.

\[\begin{CD}
{\rm weight}\, \Lambda_j-\beta   @< << {\rm core\, abacus}\,\, (\lam, j) @> >> {\rm weighted\, Uglov\, vector}\, \mathrm{u}\\
@V {\rm Theorem\,\ref{4.5}} VV                  @VV {\rm Lemma\, \ref{3.12}} V                                 @VV {\rm Lemma\, \ref{3.27}} V\\
{\rm atomic\, length}\,\mathscr{L}_{\Lambda_j}(w)         @< <<   {\rm affine\, Grassmannian\, element}\, w @> >> {\rm semidirect\, product}\, t_{\mathrm{q}}\overline{w}
\end{CD}\]

Thirdly, we establish Diophantine equations by using the height formula and develop the relevant theory of classical affine types.
In particular, we study the problem that under what conditions all solution orbits under the action of type $B$ Weyl groups are parameterized by core abaci. 
Along this way, we obtain closed formulas for computing the number of certain $j$-core abaci in some special cases, 
including affine types $C_2^{(1)}$, $D_3^{(2)}$, $B_3^{(1)}$, $D_4^{(2)}$, $B_4^{(1)}$ and $D_4^{(1)}$.

\smallskip

The paper is organized as follows. We begin in Section 2 by reviewing some preliminaries on combinatorics and affine Weyl groups.
Then in Section 3, we define and study core abaci of arbitrary charge for classical twisted and untwisted affine types. 
In particular, we establish a bijection between $j$-core abaci and the alcoves in a generalized Tits cone, which corresponds to an affine Grassmannian.
Moreover we study the compatibility between weighted Uglov vectors and $\mathrm q$-vectors.   
In Section 4, we first establish the height formula and then 
prove that the height of a core abacus $(\lam, j)$ is equal to the atomic length of $w_{\lam, j}$.
Furthermore, we will construct Diophantine equations induced by the height formula.
In Section 5, we consider the relationship between core abaci and the solutions of the Diophantine equations
obtained in Section 4. Particularly, we prove that in some small rank cases, the solutions can be completely parameterized by core abaci.
This makes us being able to give closed formulae on enumerating the number of certain $j$-core abaci in some special cases.

\bigskip


\section{Preliminaries}
In this section, we give a quick review on some preliminaries about Kac-Moody algebras, 
affine Weyl groups, abaci, Uglov map and so on. The main references are \cite{HLQ, Kac}.

\subsection{Common notations}

For $k\in \mathbb{Z}$, we write
$\mathbb{Z}_{<k}:=\{x\in\mathbb{Z}\mid x<k\}$,
and similarly for $\mathbb{Z}_{\leq k}$, $\mathbb{Z}_{>k}$ and $\mathbb{Z}_{\geq k}$.
For brevity, ${\bf 1}:=(1,\dots,1)\in \mathbb{Z}^m$ for certain $m>0$.
Of course, $\bf 1$ depends on $m$, which will always be clear from the context.
Denote by $|I|$ the number of elements in a set $I$.



\subsection{Kac-Moody algebra and affine Weyl group}

Let us fix some notations. Let $A$ be a rank $l$ {\em Cartan matrix} of an affine type.
The vector $(a_0, a_1, \dots, a_l)$ of positive integer with no common factor is defined to be in the kernel of $A$.
Similarly, $(a_0^{\vee}, a_1^{\vee}, \dots, a_l^{\vee})$ is defined by $A^T$. 
Denote by $h=\sum_i a_i$ and $h^\vee=\sum_i a_i^\vee$.
Let $(\mathfrak{h}, \Pi, \check\Pi)$ be a realization of $A$, where $\mathfrak{h}$ is a complex vector space, 
$\Pi=\{\alpha_0, \alpha_1, \dots, \alpha_l\}\subset \mathfrak{h}^*$ 
and $\check\Pi=\{\alpha_0^\vee, \alpha_1^\vee, \dots, \alpha_l^\vee\}\subset \mathfrak{h}$.
The {\em Chevalley generators} of the corresponding derived affine Lie algebra are $e_i$, $f_i$ and $h_i$, $i=0, \dots, l$.
The {\em pairing} $\langle\cdot,\cdot\rangle$ between $\mathfrak{h}$ 
and $\mathfrak{h}^*$ is given by $\langle \alpha_i^\vee, \alpha_j\rangle=a_{ij}$.
The {\em root lattice} is $Q=\oplus_i\mathbb{Z}\alpha_i$ and similarly {\em coroot lattice} $Q^\vee$. 
In particular, $\delta=\sum_ia_i\alpha_i$, $c=\sum_ia_i^\vee\alpha_i^\vee$, and $\theta=\delta-a_0\alpha_0$ is the {\em highest root}.
For each $\alpha=\sum_ik_i\alpha_i\in Q$, the {\em height} ${\rm ht}(\alpha)$ of $\alpha$ is defined to be $\sum_ik_i$. 
Particularly, we sometimes use the notation ${\rm ht}_i(\alpha)$ to denote $k_i$.
Furthermore, we denote by $\alpha>0$ if $k_i\geq 0$ for each $i$.
For $0\leq i\leq l$, define {\em fundamental weight} $\Lambda_i$ by $\langle\Lambda_i, \alpha_j^\vee\rangle=\delta_{ij}$ for $0\leq j\leq l$
and denote $\omega_i=\Lambda_i-\frac{a_i^\vee}{a_0^\vee}\Lambda_0$ for $1\leq i\leq l$. The sum of $\omega_i$ is $\rho$. 
Dually, we have {\em coweights} $\Lambda_0^\vee, \Lambda_1^\vee, \dots, \Lambda_l^\vee$ and $\omega_1^\vee, \dots, \omega_l^\vee$.
The sum of $\omega_i^\vee$ is $\rho^\vee$. The {\em affine weight lattice} is $P=\oplus_i\mathbb{Z}\Lambda_i\oplus \mathbb{Z}\delta$
and the {\em affine dominant weight lattice} is $P^+=\oplus_i\mathbb{N}\Lambda_i\oplus \mathbb{Z}\delta$.

\smallskip

Note that the names and shapes of Dynkin diagrams of affine types are not the same in different books. We now fix them for later use.

\smallskip

\begin{center}
	\begin{tabular}{lclc}
		${A}_{2l-1}^{(1)}$ & \dynkin [edge length=2.75em,extended,reverse arrows,labels={0,1,2,3,l-2,l-1,l}]B{}
		& $A_{2l}^{(2)}$ &\dynkin[edge length=2.75em,extended,labels={0,1,2,3,l-2,l-1,l}] A[2]{even}\\
\end{tabular}
\end{center}
\begin{center}
	\begin{tabular}{lclc}
		${B}_l^{(1)}$ & \dynkin [edge length=2.75em,extended,labels={0,1,2,3,l-2,l-1,l}]B{}
		& ${C}_l^{(1)}$ & \dynkin [edge length=2.75em,extended,labels={0,1,2,l-2,l-1,l}]C{}\\
\end{tabular}
\end{center}
\begin{center}
	\begin{tabular}{lclc}
		${D}_l^{(1)}$ & \dynkin [edge length=2.75em,extended,labels={0,1,2,3,l-3, l-2,l-1,l}]D{}
		& ${D}_{l+1}^{(2)}$ & \dynkin [edge length=2.75em,extended,reverse arrows,labels={0,1,2,l-2,l-1,l}]C{}\\
	\end{tabular}
\end{center}

\smallskip

It is well known that all of the Cartan matrices of affine types are symmetrizable. 
That is, for matrix $A$, there exists a diagonal matrix $D={\rm diag}(d_0, d_1, \dots, d_l)$ such that $C=DA$ is symmetric. 
In this paper, the matrices $D$ are taken as follows.

\smallskip

Type $A_{l}^{(1)}$: ${\rm diag}(1, 1, \dots , 1, 1)$; \,\,Type $A_{2l-1}^{(2)}$: ${\rm diag}(1, 1, \dots , 1, 2)$; 
Type $A_{2l}^{(2)}$: ${\rm diag}(\frac{1}{2}, 1, \dots , 1, 2)$;

\smallskip

Type $B_l^{(1)}$: ${\rm diag}(1, 1, \dots , 1, \frac{1}{2})$;\,
Type $C_l^{(1)}$: ${\rm diag}(1, \frac{1}{2}, \dots , \frac{1}{2}, 1)$;\,\, Type $D_l^{(1)}$: ${\rm diag}(1, 1, \dots , 1, 1)$;

\smallskip

Type $D_{l+1}^{(2)}$: ${\rm diag}(1, 2, \dots , 2, 1)$.

\smallskip

It is necessary to point out that for the reason of compatibility with our study, some of the matrices above are different to those in \cite{HLQ}.
Therefore, we shall use a different version of the results in \cite{HLQ} adjusted by a constant.

A {\em bilinear form} $(\cdot\, , \,\cdot)$ on $\mathfrak{h}^*$ is defined by  
$(\alpha_i, \alpha_j)=c_{ij}$, $(\Lambda_i, \alpha_j)=\delta_{ij}d_i$ for $i, j=0, 1, \dots, l$ and $(\Lambda_0, \Lambda_0)=0$.
For $\Lambda\in P$ and $\beta\in Q$, define 
$$\rm def (\Lambda-\beta)=(\Lambda, \beta)-\frac{1}{2}(\beta, \beta).$$ 

For $0 \leq i \leq l$, the {\em fundamental reflection} $\sigma_i: \mathfrak{h}^*\rightarrow \mathfrak{h}^*$ is defined by 
$$\sigma_i(v)=v-\langle v,\alpha_i^{\vee}\rangle\alpha_i$$ for $v\in \mathfrak{h}^*$. These fundamental reflections 
generate a subgroup $W$ of $\operatorname{GL}(\mathfrak{h}^*)$,
called the {\em affine Weyl group} of a Kac-Moody algebra $\mathfrak{g}(A)$. We shall use $\ell(w)$ to denote the {\em length} of $w\in W$. 
Denote by $W_0$ the finite Weyl group obtained by deleting the generator $\sigma_0$.
It is well known that $W$ can be viewed as a semidirect product of $W_0$ by $T(Q_0^\vee)$, 
where $T(Q_0^\vee)$ is the group of translations $t_{\mathrm q}$ with ${\mathrm q}\in Q_0^\vee$.
Note that any element $w\in W$ decomposes uniquely as $w=t_{\mathrm q}\overline{w}$, 
where ${\mathrm q}\in Q_0^\vee$ and $\overline{w}\in W_0$.

\smallskip

{\em Atomic length} is a new statistic on affine Weyl groups introduced by Chapelier-Laget and Gerber.
\begin{definition}\cite[Definition 4.5]{CG}\label{2.1}
Let $W$ be an affine Weyl group and let $\Lambda$ be a dominant weight. For all $w\in W$,
$$\mathscr{L}_{\Lambda}(w)=\langle\Lambda-w(\Lambda),\rho^{\vee}\rangle$$
\end{definition}

\smallskip


\subsection{$l$-index system}
In this subsection, we define the so-called {\em $l$ index system} for later use. 
For a given affine type  $X$, define 
a set $I_X$ according to the Dynkin diagram of $X$, 
$$I_X=\begin{cases}
        \{\pm 1, \dots, \pm l\}, & \mbox{if }  0 \nLeftarrow 1, l-1 \nRightarrow l;\\
        \{0, \pm 1, \dots, \pm l\}, & \mbox{if } 0 \Leftarrow 1,  l-1 \nRightarrow l;\\
       \{\pm 1, \dots, \pm l, l+1\}, & \mbox{if }  0 \nLeftarrow 1,  l-1 \Rightarrow l;\\
        \{0, \pm 1, \dots, \pm l, l+1\}, & \mbox{if } 0 \Leftarrow 1,  l-1 \Rightarrow l,
      \end{cases}$$
and a map $\iota: \{0, 1, \dots, |I_X|-1\}\rightarrow I_X$:     
$$\iota(x)=\begin{cases}
             x+1, & \mbox{if } 0\leq x \leq  l-1\, {\rm and}\, l+1\notin I_X, \,{\rm or}\\
             &  \,\,\,\,\,\,0\leq x \leq  l\, {\rm and}\, l+1\in I_X;\\
             x-2l, & \mbox{if } l\leq x \leq |I_X|-1\, {\rm and}\, l+1\notin I_X;\\
             x-2l-1, & \mbox{if } l+1\leq x\leq  |I_X|-1 \,{\rm and}\, l+1\in I_X.
           \end{cases}$$     
     
\begin{definition}\label{2.13}
For a given affine type $X$ and an integer $x\in\mathbb{Z}$, suppose that $x=q'|I_X|+r'$ with $r'\in \{0, 1, \dots, |I_X|-1\}$. 
Then the $X$ type $l$ index of $x$ is defined to be a pair $(\iota(r'),q)\in I_X\times \mathbb{Z}$, where $q$ is defined as follows: 
$$q=\begin{cases}
      2q', & \mbox{\rm if } \iota(r')=r'+1; \\
      2q'+1, & \mbox{\rm otherwise}.
    \end{cases}$$
Sometimes we call $\iota(r')$ the $l$ index of $x$ if there is no danger of confusion. 
Particularly, if $(r, k)$ is not an $l$ index of any integer $x$ for arbitrary $k\in \mathbb{Z}$, we say $r$ is not an $l$ index.
\end{definition}

\smallskip

\subsection{$\beta$-set, partition and abacus}
Let $n$ be a nonnegative integer. A {\em partition} $\lam$ of $n$ is a non-increasing sequence of
nonnegative infinite integers $\lam=(\lam_1,\lam_2,\dots)$ such that
$\sum_{i}\lam_i=n$. We write $|\lam|=n$ and call $n$ the {\em size} of $\lam$. 
The partition of 0 is usually denoted by $\varnothing$.
The {\em Young diagram} of a
partition $\lam$ is the set of nodes $[\lam]=\{(i,j)\mid 1\leq i,
1\leq j\leq\lam_{i}\}$.
The {\em conjugate} of $\lam$ is defined to be a partition
$\lam'=(\lam'_1,\lam'_2,\cdots)$, where $\lam'_j$ is equal to the
number of nodes in column $j$ of $[\lam]$ for $j=1,2,\cdots$.

Recall that a {\em $\beta$-set} $\B$ is a subset of $\mathbb{Z}$ such that both $\max(\B)$ and $\min(\mathbb{Z} \setminus \B)$ exist. 
Its charge $s(\B)$ is defined as $$s(\B) := |\B \cap \mathbb{Z}_{\geq 0}| - |\mathbb{Z}_{<0} \setminus \B|.$$ 
Given a horizontal line with positions labelled by $\mathbb{Z}$ in an
ascending order going from left to right, we can obtain the abacus display of 
a $\beta$-set by putting a bead at the position $x$ for each $x \in \B$.
Note that we always draw a dashed vertical line between positions $-1$ and $0$.
A position without a bead will be called empty. Moreover, we say {\em reverse the color} of a position if we replace the bead by empty or vice versa.

Given a partition $\lam$ and $j\in \mathbb{Z}$, one can associate it
to a $\beta$-set defined by $$\beta_j(\lam)=\{\lam_i-i+j\mid i\in\mathbb{N}^+\}.$$
On the other hand, it is clear that a $\beta$-set uniquely determines a partition.
We often write the abacus display of $\beta_j(\lam)$ as a pair $(\lam, j)$.
Let us illustrate an example.

\begin{example}\label{2.2}\rm{
Let $\lam=(7, 5, 4, 1, 1)$ and $j=0$. For each $i\in \beta_0(\lam)$,
we set a bead at the $i$-th position on the horizontal abacus.
Then $(\lam, 0)$ is expressed as below. 
\[
\begin{tikzpicture}
[scale=0.5, bb/.style={draw,circle,fill,minimum size=2.5mm,inner sep=0pt,outer sep=0pt},
wb/.style={draw,circle,fill=white,minimum size=2.5mm,inner sep=0pt,outer sep=0pt}]
\foreach \x in {13,-10}
\foreach \y in {2}
{
\node at (\x,\y) {$\cdots$};}	
\node [wb] at (12,2) {};
\node [wb] at (11,2) {};
\node [wb] at (10,2) {};
\node [wb] at (9,2) {};
\node [bb] at (8,2) {};
\node [wb] at (7,2) {};
\node [wb] at (6,2) {};
\node [bb] at (5,2) {};
\node [wb] at (4,2) {};
\node [bb] at (3,2) {};
\node [wb] at (2,2) {};
\node [wb] at (1,2) {};
\node [wb] at (0,2) {};
\node [bb] at (-1,2) {};
\node [bb] at (-2,2) {};
\node [wb] at (-3,2) {};
\node [bb] at (-4,2) {};
\node [bb] at (-5,2) {};
\node [bb] at (-6,2) {};
\node [bb] at (-7,2) {};
\node [bb] at (-8,2) {};
\node [bb] at (-9,2) {};
\draw[dashed](1.5,1)--node[]{}(1.5,2.5);
\end{tikzpicture}
\]}
\end{example}  

The $\beta$-sets defined above will be called {\bf two-sided $\beta$-sets} 
because we need to generalize the notion of $\beta$-set to include the so-called {\bf one-sided} case. 

\begin{definition}\label{2.4}
Let $k\in \mathbb{Z}$. 
A one-sided $\beta$-set $\B_{\geq k}$ is a subset of $\mathbb{Z}_{\geq k}$ such that $\max(\B_{\geq k})$ exists.
The notation $B_{\geq -\infty}$ denotes a two-sided $\beta$-set, where the subscript ``\,$\geq -\infty$" will be often omitted if no confusion occurs.
\end{definition}

Similarly, one can give the so-called half $k$ abacus display of a one sided $\beta$-set 
in a horizontal ray with positions labelled by integers not less than $k$. 
To keep the notation simple, we usually write a half $k$-abaci as a pair $(\lam, k)$. 
However, the notation $\lam$ can not be treated as a partition in this case.
To distinguish, the abacus corresponding to a two-sided $\beta$-set will be called a whole abacus.
Both whole abaci and half $k$ ones are called abaci. Denote the set of all abaci by $\mathcal{B}$. 
We often do not distinguish abaci and $\beta$-sets if no confusion occurs.

\begin{example}\label{2.5}
Let $\B_{\geq 0}=\{0,3,5,7,8,10\}$. Then its half abacus is as follows.
\[
\begin{tikzpicture}[scale=0.7, every node/.style={transform shape}]
			\begin{scope}[yshift=0cm]
				\fill[black] (0.3,0) circle (0.2cm);
				\draw[black] (0.9,0) circle (0.2cm);
				\draw[black] (1.5,0) circle (0.2cm);
				\fill[black] (2.1,0) circle (0.2cm);
				\draw[black] (2.7,0) circle (0.2cm);
				\fill[black] (3.3,0) circle (0.2cm);
				\draw[black] (3.9,0) circle (0.2cm);
				 \fill[black](4.5,0) circle (0.2cm);
				\fill[black] (5.1,0) circle (0.2cm);
				\draw[black] (5.7,0) circle (0.2cm);
				\fill[black] (6.3,0) circle (0.2cm);
				\draw[black] (6.9,0) circle (0.2cm);
				\draw[black] (7.5,0) circle (0.2cm);
				\node at (8.1,0) {$\dots$};
			\end{scope}
\end{tikzpicture}
\]
\end{example}

In order to connect abaci to affine Kac-Moody algebras, we need to define the type for them by $l$ index system.
\begin{definition}
An abacus is called to be of affine type $X$ if we take the $X$ type $l$ index for each position. 
\end{definition}

\begin{example}\label{2.14} \rm{
Let $(\lam, j)=((5,2,1,1,1,1,1), 1)$. Then the $D_3^{(2)}$ type $l$-index of each position is illustrate in the following diagram.
\[
\begin{tikzpicture}[scale=0.7, every node/.style={transform shape}]
\begin{scope}[yshift=0cm]
\node [] at (9,-0.6) {$0$};
\node [] at (8,-0.6) {$-1$};
\node [] at (7,-0.6) {$-2$};
\node [] at (6,-0.6) {$3$};
\node [] at (5,-0.6) {$2$};
\node [] at (4,-0.6) {$1$};
\node [] at (3,-0.6) {$0$};
\node [] at (10,-0.6) {$1$};
\node [] at (11,-0.6) {$2$};
\node [] at (12,-0.6) {$3$};
\node [] at (13,-0.6) {$-2$};
\node [] at (14,-0.6) {$-1$};
\node [] at (15,-0.6) {$0$};
\node [] at (16,-0.6) {$1$};
\node at (2,0) {$\dots$};
\node at (2,-0.5) {$\dots$};
\fill[black] (3.0,0) circle (0.2cm);
\draw[black](4,0) circle (0.2cm);
\fill[black] (5,0) circle (0.2cm);
\fill[black] (6,0) circle (0.2cm);
\fill[black] (7,0) circle (0.2cm);
\fill[black] (8,0) circle (0.2cm);
\fill[black] (9,0) circle (0.2cm);
\draw[black] (10,0) circle (0.2cm);
\fill[black](11,0) circle (0.2cm);
\draw[black] (12,0) circle (0.2cm);
\draw[black] (13,0) circle (0.2cm);
\draw[black] (14,0) circle (0.2cm);
\fill[black] (15,0) circle (0.2cm);
\draw[black] (16,0) circle (0.2cm);
\draw[dashed] (9.5,-0.8)--(9.5,0.6);
\node at (17,0) {$\dots$};
\node at (17,-0.5) {$\dots$};
\end{scope}
\end{tikzpicture}
\]}
\end{example}

\begin{remark}
By Definition \ref{2.13}, the $l$ index systems of affine type $A^{(2)}_{2l-1}$, $C^{(1)}_l$, $D^{(1)}_l$ are the same.
\end{remark}



\subsection{Actions of $f_i$ and affine Weyl group}
Given an abacus $(\lam, j)$ of affine type $X$, if $(i, 2k)$ has a bead and $(i+1, 2k)$ is empty, 
or if $(-i-1, 2k+1)$ has a bead and $(-i, 2k+1)$ is empty for $0<i<l$ and $k\in \mathbb{Z}$,
then we say moving the bead to the empty position an {\em $f_i$-action} on $(\lam, j)$. 

Moreover, the $f_0$ and $f_l$ actions on a whole abacus are defined in Table 1. 
Note that the affine type $A_l^{(1)}$ is not contained because we almost have no new result about this type.

\begin{table}[H]
\centering
\caption{Action of $f_0$ and $f_l$}
\begin{tabular}{|c|c|c|}
\hline
\textbf{Moving} & \textbf{Type} & \textbf{Action} \\
\hline
\multirow{2}{*}{\((-1, 2k-1)\rightarrow (1, 2k)\)} & \(C_l^{(1)}\) & \(f_0\) \\
\cline{2-3}
& \(A_{2l}^{(2)}, D_{l+1}^{(2)}\) & \(2f_0\) \\
\hline
{\((-1, 2k-1)\rightarrow (2, 2k)\)} & \multirow{2}{*}{\(A_{2l-1}^{(2)}, B_l^{(1)}, D_l^{(1)}\)} & \multirow{2}{*}{\(f_0\)} \\
			\cline{1-1}
			 \((-2, 2k-1)\rightarrow (1, 2k)\) &  & \\
\hline
\multirow{2}{*}{\((l, 2k)\rightarrow (-l, 2k+1)\)} & \(A_{2l-1}^{(2)}, A_{2l}^{(2)}, C_l^{(1)}\) & \(f_l\) \\
\cline{2-3}
& \(B_{l}^{(1)}, D_{l+1}^{(2)}\) & \(2f_l\) \\
\hline
{\((l, 2k)\rightarrow (1-l, 2k+1)\)} & \multirow{2}{*}{\(D_l^{(1)}\)} & \multirow{2}{*}{\(f_l\)} \\
			\cline{1-1}
			 \((l-1, 2k)\rightarrow (-l, 2k+1)\) &  & \\
\hline
		\end{tabular}
\end{table}

For a half $k$ abacus, besides the movings given in the Table 1, the $f_0$ and $f_l$ actions also include 
some special ones defined in Table 2.

\begin{table}[H]
\centering
\caption{Special actions of $f_0$ and $f_l$}
\begin{tabular}{|c|c|c|}
\hline
\textbf{Moving} & \textbf{Type} & \textbf{Action} \\
\hline
\(\bigcirc_{(1, 0)} \rightarrow {\rm bead} \) & \(A_{2l}^{(2)},  D_{l+1}^{(2)}\) & \(f_0\) \\
\hline
 \(\bigcirc_{(1, 0)}, \bigcirc_{(2, 0)}\rightarrow {\rm beads}\) & \(A_{2l-1}^{(2)}, B_{l}^{(1)}, D_l^{(1)}\) & \(f_0\) \\
\hline
\(\bigcirc_{(-l, 1)} \rightarrow {\rm bead} \) & \(B_{l}^{(1)},  D_{l+1}^{(2)}\) & \(f_l\) \\
\hline
 \(\bigcirc_{(-l, 1)}, \bigcirc_{(1-l, 1)}\rightarrow {\rm beads}\) & \(D_l^{(1)}\) & \(f_l\) \\
\hline
\end{tabular}
\end{table}

The affine Weyl group $W$ acts on abaci naturally. The action can be described by using $f_i$.

\begin{definition}\label{2.7}
Given an abacus $(\lam, j)$ of affine type $X$ and $0\leq i\leq l$, 
the abacus $\sigma_i(\lam, j)$ is obtained by doing all possible $f_i$-actions on $(\lam, j)$,
where $\sigma_i$ is a generator of $W$.
\end{definition}

To end this subsection, we connect abaci to the data of Kac-Moody algebras. Define the {\em abacus of weight $\Lambda_j$} of affine type $X$ as follows.
If $X=A_l^{(1)}$ or $C_l^{(1)}$, or the type $X$ satisfies $\frac{a_j^\vee}{a_0^\vee}=2$, then for $0\leq j\leq l$ the abacus of $\Lambda_j$ is the whole abacus $(\varnothing, j)$.
If $\Lambda_0$ is not affine type $A_l^{(1)}$ and $C_l^{(1)}$, then the abacus of $\Lambda_0$ corresponds to the one-sided $\beta$-set $B_{\geq 0}=\varnothing$.
If $\Lambda_1$ is of affine type $A_{2l-1}^{(2)}$, $B_l^{(1)}$ or $D_l^{(1)}$, then the abacus of $\Lambda_1$ corresponds to the one-sided $\beta$-set $B_{\geq 0}=\{0\}$.
If $\Lambda_{l-1}$ is of affine type $D_l^{(1)}$, then the abacus of $\Lambda_{l-1}$ corresponds to the one-sided $\beta$-set $B_{\geq l}=\{l\}$.
If $\Lambda_{l}$ is of affine type $B_l^{(1)}$ or $ D_{l+1}^{(2)}$, then the abacus of $\Lambda_{l}$ corresponds to the one-sided $\beta$-set $B_{\geq l+1}=\varnothing$.

\begin{definition}\label{2.6}
Let $0\leq j\leq l$. If an abacus $(\lam, j)$ is obtained from the abacus of $\Lambda_j$ of affine type $X$ by $f_i$-actions, 
we say $(\lam, j)\in \Lambda_j-\beta$, where the coefficient of $\alpha_i$ in $\beta$ is the total number of 
$f_i$-actions from the abacus of $\Lambda_j$ to $(\lam, j)$.
\end{definition}

\smallskip


\subsection{Shifted Young diagram}
Given a Young diagram $[\lam]$ and an integral vector $v$ with $v_i\geq 0$, the {\em shifted Young diagram} with shift vector $v$ of $[\lam]$ is
obtained by shifting the $i$-th row in $[\lam]$ to the right by $v_i$ squares for all $i \geq 1$. 

\begin{definition}\label{shifted Young diagram}
Given a one-sided $\beta$-set $B_{\geq k}=\{ a_1> a_2> \dots> a_m\}$ with $k\in \mathbb{Z}$ and a vector $v=(v_1, v_2, \dots, v_n)$, 
one can connect it to the shifted Young diagram with shift vector $v$ of $[\lam]$, where 
$$\lam_i= a_i-k+i+v_1-v_i$$ for $1\leq i\leq n$. 
\end{definition}

\begin{example}\label{2.10}{\rm
Take $\B_{\geq 0}=\{0,3,5,7,8,10\}$ which is the same as that in Example \ref{2.5} and $v=(0, 2, 2, 4, 4)$. 
Then $\lam=(11, 8, 8, 5, 4)$ and the corresponding shifted Young diagram is 
\[
	\scalebox{0.35}{  
		\begin{tikzpicture}[x=1.5cm, y=1.5cm]  
			\node[emptycell] at (1,0) {};
			\node[emptycell] at (2,0) {};
			\node[emptycell] at (3,0) {};
			\node[emptycell] at (4,0) {};
			\node[emptycell] at (5,0) {};
			\node[emptycell] at (6,0) {};
			\node[emptycell] at (7,0) {};
			\node[emptycell] at (8,0) {};
			\node[emptycell] at (9,0) {};
			\node[emptycell] at (10,0) {};
			\node[emptycell] at (11,0) {};
			\node[emptycell] at (3,-1) {};
			\node[emptycell] at (4,-1) {};
			\node[emptycell] at (5,-1) {};
			\node[emptycell] at (6,-1) {};
			\node[emptycell] at (7,-1) {};
			\node[emptycell] at (8,-1) {};
			\node[emptycell] at (9,-1) {};
			\node[emptycell] at (10,-1) {};
			\node[emptycell] at (3,-2) {};
			\node[emptycell] at (4,-2) {};
			\node[emptycell] at (5,-2) {};
			\node[emptycell] at (6,-2) {};
			\node[emptycell] at (7,-2) {};
			\node[emptycell] at (8,-2) {};
			\node[emptycell] at (9,-2) {};
			\node[emptycell] at (10,-2) {};
			\node[emptycell] at (5,-3) {};
			\node[emptycell] at (6,-3) {};
			\node[emptycell] at (7,-3) {};
			\node[emptycell] at (8,-3) {};
			\node[emptycell] at (9,-3) {};
			\node[emptycell] at (5,-4) {};
			\node[emptycell] at (6,-4) {};
			\node[emptycell] at (7,-4) {};
			\node[emptycell] at (8,-4) {};	
	\end{tikzpicture}}
\]}
\end{example}

For a one-sided $\beta$-set $\B_{\geq k}$ of affine type $X$, we can also associate a two-sided $\beta$-set $\B$ with it.
\begin{definition}\label{2.8}
Let $\B_{\geq k}$ be a one-sided $\beta$-set of affine type $X$. Then the associated $\beta$-set is defined to be
$$B:=\begin{cases}
	B_{\geq k} \bigcup(\mathbb{Z}_{<k}-\{2k-1-x\mid x\in B_{\geq k} \}), &  \text{\rm if } 0\nLeftarrow 1, k=0,l; \\
	B_{\geq k} \bigcup(\mathbb{Z}_{<k}-\{2k-2-x\mid x\in B_{\geq k} \}),  &  \text{\rm if } 0\Leftarrow 1,k=0; \\
	B_{\geq k} \bigcup(\mathbb{Z}_{<k}-\{2k-2-x\mid x\in B_{\geq k} \}-\{l\}),  &  \text{\rm if } l-1\Rightarrow l,k=l+1,
\end{cases}$$
where ``\,$0\Leftarrow 1$" means that in the Dynkin diagram of $X$, there is a ``$\Leftarrow$" from vertex 1 to vertex 0, 
and the meaning of the others is similar.
\end{definition}

\begin{example}\label{2.9}\rm{
Let $B_{\geq 0}=\{0,3,5,7,8,10\}$ be a $\beta$-set of affine type $D^{(1)}$.  
Then the abacus of the associated two-sided $\beta$-set is as follows.
\[		
\begin{tikzpicture}[scale=0.7, every node/.style={transform shape}]
		\begin{scope}[yshift=0cm]
			\draw[black] (-0.3,0) circle (0.2cm);
			\fill[black] (-0.9,0) circle (0.2cm);
			\fill[black] (-1.5,0) circle (0.2cm);
			\draw[black](-2.1,0) circle (0.2cm);
			\fill[black] (-2.7,0) circle (0.2cm);
			\draw[black](-3.3,0) circle (0.2cm);
			\fill[black] (-3.9,0) circle (0.2cm);
			\draw[black](-4.5,0) circle (0.2cm);
			\draw[black](-5.1,0) circle (0.2cm);
			\fill[black] (-5.7,0) circle (0.2cm);
			\draw[black](-6.3,0) circle (0.2cm);
			\fill[black] (-6.9,0) circle (0.2cm);
			\fill[black] (-7.5,0) circle (0.2cm);
			\node at (-8.1,0) {$\dots$};
			\fill[black] (0.3,0) circle (0.2cm);
			\draw[black] (0.9,0) circle (0.2cm);
			\draw[black] (1.5,0) circle (0.2cm);
			\fill[black] (2.1,0) circle (0.2cm);
			\draw[black] (2.7,0) circle (0.2cm);
			\fill[black] (3.3,0) circle (0.2cm);
			\draw[black] (3.9,0) circle (0.2cm);
			\fill[black](4.5,0) circle (0.2cm);
			\fill[black] (5.1,0) circle (0.2cm);
			\draw[black] (5.7,0) circle (0.2cm);
			\fill[black] (6.3,0) circle (0.2cm);
			\draw[black] (6.9,0) circle (0.2cm);
			\draw[black] (7.5,0) circle (0.2cm);
			\draw[dashed] (0,-0.3)--(0,0.3);
			\node at (8.1,0) {$\dots$};
		\end{scope}
	\end{tikzpicture}
\]}
\end{example}

Given a shifted Young diagram $Y$ of $[\lam]$ obtained from a one sided $\beta$-set $B_{\geq k}$ satisfying 
the second (or third) condition in Definition \ref{2.8} with shift vector $(0, 1, 2, \dots)$ (or $(1, 2, \dots)$), 
define $[\ddot{\lam}]$ to be the Young diagram of the {\em double distinct partition} 
(see \cite[Chapter III]{M} for details) $\ddot{\lam}$ obtained from $\lam$.
If a shifted Young diagram of $[\lam]$ obtained from a one sided $\beta$-set satisfying 
the first condition in Definition \ref{2.8} with shift vector $(0, 2, 2, 4, 4, \dots, 2k, 2k)$, 
define $[\ddot{\lam}]$ to be the Young diagram $$Y\cup \{(i, i)\mid 1\leq i\leq 2k+2\}\cup \{(j, i)\mid (i, j)\in Y\}.$$
We call $[\ddot{\lam}]$ the Young diagram of $B_{\geq k}$.

\begin{example}\label{2.11}
Let $Y$ be the shifted Young diagram in Example \ref{2.10}. Then the corresponding Young diagram is
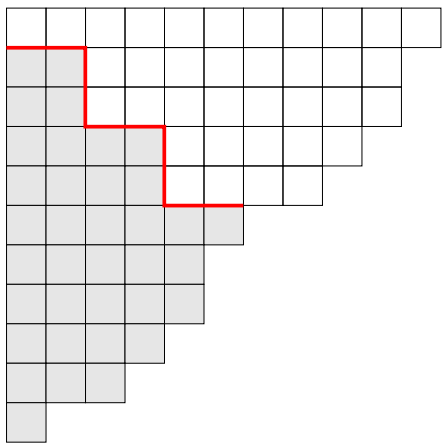
\[
	\scalebox{0.35}{  
		\begin{tikzpicture}[x=1.5cm, y=1.5cm]  
			\node[emptycell] at (1,0) {};
			\node[emptycell] at (2,0) {};
			\node[emptycell] at (3,0) {};
			\node[emptycell] at (4,0) {};
			\node[emptycell] at (5,0) {};
			\node[emptycell] at (6,0) {};
			\node[emptycell] at (7,0) {};
			\node[emptycell] at (8,0) {};
			\node[emptycell] at (9,0) {};
			\node[emptycell] at (10,0) {};
			\node[emptycell] at (11,0) {};
			\node[graycell] at (1,-1) {};
			\node[graycell] at (2,-1) {};
			\node[emptycell] at (3,-1) {};
			\node[emptycell] at (4,-1) {};
			\node[emptycell] at (5,-1) {};
			\node[emptycell] at (6,-1) {};
			\node[emptycell] at (7,-1) {};
			\node[emptycell] at (8,-1) {};
			\node[emptycell] at (9,-1) {};
			\node[emptycell] at (10,-1) {};
			\node[graycell] at (1,-2) {};
			\node[graycell] at (2,-2) {};
			\node[emptycell] at (3,-2) {};
			\node[emptycell] at (4,-2) {};
			\node[emptycell] at (5,-2) {};
			\node[emptycell] at (6,-2) {};
			\node[emptycell] at (7,-2) {};
			\node[emptycell] at (8,-2) {};
			\node[emptycell] at (9,-2) {};
			\node[emptycell] at (10,-2) {};
			\node[graycell] at (1,-3) {};
			\node[graycell] at (2,-3) {};
			\node[graycell] at (3,-3) {};
			\node[graycell] at (4,-3) {};
			\node[emptycell] at (5,-3) {};
			\node[emptycell] at (6,-3) {};
			\node[emptycell] at (7,-3) {};
			\node[emptycell] at (8,-3) {};
			\node[emptycell] at (9,-3) {};
			\node[graycell] at (1,-4) {};
			\node[graycell] at (2,-4) {};
			\node[graycell] at (3,-4) {};
			\node[graycell] at (4,-4) {};
			\node[emptycell] at (5,-4) {};
			\node[emptycell] at (6,-4) {};
			\node[emptycell] at (7,-4) {};
			\node[emptycell] at (8,-4) {};
			\node[graycell] at (1,-5) {};
			\node[graycell] at (2,-5) {};
			\node[graycell] at (3,-5) {};
			\node[graycell] at (4,-5) {};
			\node[graycell] at (5,-5) {};
			\node[graycell] at (6,-5) {};
			\node[graycell] at (1,-6) {};
			\node[graycell] at (2,-6) {};
			\node[graycell] at (3,-6) {};
			\node[graycell] at (4,-6) {};
			\node[graycell] at (5,-6) {};
			\node[graycell] at (1,-7) {};
			\node[graycell] at (2,-7) {};
			\node[graycell] at (3,-7) {};
			\node[graycell] at (4,-7) {};
			\node[graycell] at (5,-7) {};
			\node[graycell] at (1,-8) {};
			\node[graycell] at (2,-8) {};
			\node[graycell] at (3,-8) {};
			\node[graycell] at (4,-8) {};
			\node[graycell] at (1,-9) {};
			\node[graycell] at (2,-9) {};
			\node[graycell] at (3,-9) {};
			\node[graycell] at (1,-10) {};
	 \draw[red, line width=4pt, -] 
	(1,0) -- (3,0) -- (3,-2) -- (5,-2) -- (5,-4) -- (7,-4);		
	\end{tikzpicture}}
\]
\end{example}

The following lemma reveals the relation between the two-sided $\beta$-set in Definition \ref{2.8} and the partition $\ddot{\lam}$.

\begin{lemma}\label{2.12}
Let $B_{\geq k}$ be a one sided $\beta$-set with $B$ its associated two-sided $\beta$-set and $[\ddot{\lam}]$ its Young diagram.
Then the partition determined by $B$ is $\ddot{\lam}$.
\end{lemma}

\begin{proof}
It can be proved case by case by routine calculation.
\end{proof}

\smallskip


\subsection{Residue}

A whole abacus clearly corresponds to a Young diagram with a charge. 
In Definition \ref{2.8}, we define for a half abacus the associated shifted Young diagram.
Given two pairs $(\lam, j)$ and $(\mu, j)$, where $(\mu, j)$ is obtained from $(\lam, j)$
by an $f_i$ (or $2f_i$) action, then the {\em residues} of nodes $[\mu]-[\lam]$ are defined to be $i$.
In most cases, the residue of a node in a Young diagram or a shifted one is completely determined by 
the node itself. However, this does not holds for branching cases. The main reason is that we can reach 
a node by difference paths. The phenomenon has been considered in \cite{HJ2} by combinatorial method. 
The explain given here is by using Lie theory. Let us illustrate an example below.

\begin{example}\label{2.21}
\rm{Let $((5), 2)$ be a core abacus of type $D_5^{(1)}$. Then the associated affine Weyl group element is $\sigma_5\sigma_4\sigma_3\sigma_2$.
The abacus $\sigma_3\sigma_2((5), 2)$ is

\[
\begin{tikzpicture}[scale=0.7, every node/.style={transform shape}]

\begin{scope}[yshift=0cm]
\node at (3.0,0) {$\dots$};
\fill[black] (4,0) circle (0.2cm);
\fill[black] (5,0) circle (0.2cm);
\fill[black] (6,0) circle (0.2cm);
\fill[black] (7,0) circle (0.2cm);
\draw[black] (8,0) circle (0.2cm);
\draw[black] (9,0) circle (0.2cm);
\fill[black] (10,0) circle (0.2cm);
\draw[black] (11,0) circle (0.2cm);
\draw[dashed] (6.5,-0.3)--(6.5,0.3);
\node at (12,0) {$\dots$};
\end{scope}
\end{tikzpicture}
\]
Next, $\sigma_4\sigma_3\sigma_2((5), 2)$ is
\[
\begin{tikzpicture}[scale=0.7, every node/.style={transform shape}]

\begin{scope}[yshift=0cm]
\node at (3.0,0) {$\dots$};
\fill[black] (4,0) circle (0.2cm);
\fill[black] (5,0) circle (0.2cm);
\fill[black] (6,0) circle (0.2cm);
\fill[black] (7,0) circle (0.2cm);
\draw[black] (8,0) circle (0.2cm);
\draw[black] (9,0) circle (0.2cm);
\draw[black](10,0) circle (0.2cm);
\fill[black] (11,0) circle (0.2cm);
\draw[black](12,0) circle (0.2cm); 
\draw[dashed] (6.5,-0.3)--(6.5,0.3);
\node at (13,0) {$\dots$};
\end{scope}
\end{tikzpicture}
\]
And finally,  $\sigma_5\sigma_4\sigma_3\sigma_2((5), 2)$ is
\[
\begin{tikzpicture}[scale=0.7, every node/.style={transform shape}]			

\begin{scope}[yshift=0cm]
\node at (3.0,0) {$\dots$};
\fill[black] (4,0) circle (0.2cm);
\fill[black] (5,0) circle (0.2cm);
\fill[black] (6,0) circle (0.2cm);
\fill[black] (7,0) circle (0.2cm);
\draw[black] (8,0) circle (0.2cm);
\draw[black] (9,0) circle (0.2cm);
\draw[black](10,0) circle (0.2cm);
\draw[black](11,0) circle (0.2cm);
\draw[black](12,0) circle (0.2cm); 
\fill[black](13,0) circle (0.2cm); 
\draw[black](14,0) circle (0.2cm); 
\draw[dashed] (6.5,-0.3)--(6.5,0.3);
\node at (15,0) {$\dots$};
\end{scope}
\end{tikzpicture}
\]
The corresponding Young diagrams with residues are
$$\begin{array}{ccccc}
  \young(23) & \rightarrow & \young(234) & \rightarrow & \young(23455) 
\end{array}$$
On the other hand,  $\sigma_5\sigma_3\sigma_2((5), 2)$ is
\[
\begin{tikzpicture}[scale=0.7, every node/.style={transform shape}]
			
\begin{scope}[yshift=0cm]
\node at (3.0,0) {$\dots$};
\fill[black] (4,0) circle (0.2cm);
\fill[black] (5,0) circle (0.2cm);
\fill[black] (6,0) circle (0.2cm);
\fill[black] (7,0) circle (0.2cm);
\draw[black] (8,0) circle (0.2cm);
\draw[black] (9,0) circle (0.2cm);
\draw[black](10,0) circle (0.2cm);
\draw[black](11,0) circle (0.2cm);
\fill[black] (12,0) circle (0.2cm); 
\draw[black](13,0) circle (0.2cm); 
\draw[dashed] (6.5,-0.3)--(6.5,0.3);
\node at (14,0) {$\dots$};
\end{scope}
\end{tikzpicture}\\
\]
And then we get the final abacus by acting $\sigma_4$.
Clearly, the Young diagrams with residue are
$$\begin{array}{ccccc}
  \young(23) & \rightarrow & \young(2355) & \rightarrow & \young(23554) 
\end{array}$$

Let us consider another core abacus $((2,1,1,1), 2)$. 
The associated affine Weyl group element is $\sigma_0\sigma_1\sigma_3\sigma_2$.
We first get $\sigma_3\sigma_2((2,1,1,1), 2)$: 
\[
\begin{tikzpicture}[scale=0.7, every node/.style={transform shape}]
\begin{scope}[yshift=0cm]
\node at (3.0,0) {$\dots$};
\fill[black] (4,0) circle (0.2cm);
\fill[black] (5,0) circle (0.2cm);
\fill[black] (6,0) circle (0.2cm);
\fill[black] (7,0) circle (0.2cm);
\draw[black] (8,0) circle (0.2cm);
\draw[black] (9,0) circle (0.2cm);
\fill[black] (10,0) circle (0.2cm);
\draw[black] (11,0) circle (0.2cm);
\draw[dashed] (6.5,-0.3)--(6.5,0.3);
\node at (12,0) {$\dots$};
\end{scope}
\end{tikzpicture}
\]
Then $\sigma_1\sigma_3\sigma_2((2,1,1,1), 2)$ is
\[
\begin{tikzpicture}[scale=0.7, every node/.style={transform shape}]
\begin{scope}[yshift=0cm]
\node at (3.0,0) {$\dots$};
\fill[black] (4,0) circle (0.2cm);
\fill[black] (5,0) circle (0.2cm);
\fill[black] (6,0) circle (0.2cm);
\draw[black](7,0) circle (0.2cm);
\fill[black]  (8,0) circle (0.2cm);
\draw[black] (9,0) circle (0.2cm);
\fill[black] (10,0) circle (0.2cm);
\draw[black] (11,0) circle (0.2cm);
\draw[dashed] (6.5,-0.3)--(6.5,0.3);
\node at (12,0) {$\dots$};
\end{scope}
\end{tikzpicture}
\]
and the final abacus is
\[
\begin{tikzpicture}[scale=0.7, every node/.style={transform shape}]
\begin{scope}[yshift=0cm]
\node at (3.0,0) {$\dots$};
\fill[black] (4,0) circle (0.2cm);
\draw[black](5,0) circle (0.2cm);
\fill[black] (6,0) circle (0.2cm);
\fill[black](7,0) circle (0.2cm);
\fill[black]  (8,0) circle (0.2cm);
\draw[black] (9,0) circle (0.2cm);
\fill[black] (10,0) circle (0.2cm);
\draw[black] (11,0) circle (0.2cm);
\draw[dashed] (6.5,-0.3)--(6.5,0.3);
\node at (12,0) {$\dots$};
\end{scope}
\end{tikzpicture}
\]
The corresponding Young diagrams with residues are
$$\begin{array}{ccccc}
\young(23) & \rightarrow & \young(23,1) & \rightarrow & \young(23,1,0,0) 
\end{array}$$
If we choose another path to obtain the terminal abacus, the corresponding Young diagrams with residues are
$$\begin{array}{ccccc}
\young(23) & \rightarrow & \young(23,0,0) & \rightarrow & \young(23,0,0,1) 
\end{array}$$}
\end{example}

\medskip


\subsection{Uglov map}
In \cite{U}, Uglov defined a map (called Uglov map) ${\rm\bf U}: \mathcal{B}^r\rightarrow \mathcal{B}^e$ for affine type $A_{e-1}^{(1)}$.
Recently, the first named author and his collaborators \cite{HLQ} generalized it to all classical affine types.
The goal of this subsection is to recall the definition of Uglov map, which play a key role in this paper.

The definition of Uglov map is based on set $|I_X|$ and the index system given in Definition \ref{2.13}.

\begin{definition}\label{2.15}
Given an abacus $(\lam, j)$ of affine type $X$, the image ${\rm\bf U}(\lam, j)$ under the Uglov map can be obtained by the following three steps.
\begin{enumerate}
\item [\rm(1)]\,Rotate the negative part toward below such that the position $(i, -k)$ under $(-i, k-1)$, $1\leq i\leq l$ 
and $(l+1, -k)$ (if there exists) under $(l+1, k-2)$, where $k>0$. Then reverse the color of the second row (If the abacus has not negative part, then omit this step).
Denote the resulted diagram by ${\rm\bf U}_1(\lam, j)$.
\item [\rm(2)]\,Cut up ${\rm\bf U}_1(\lam, j)$ into sections with every $|I_X|$ columns 
from left to right 
and then put them to below the first section one by one.
Denote the resulted diagram by ${\rm\bf U}_2(\lam, j)$.
\item [\rm(3)]\,Reverse the color of the $l$ columns on the right in ${\rm\bf U}_2(\lam, j)$ and rotate them toward above such that 
the column labeled by $-i$ above the column labeled by $i$ for $1\leq i\leq l$.
\end{enumerate}
\end{definition}

The best way to understand $\rm\bf U$ is via an example.

\begin{example}\label{2.16}{\rm
Take the pair the same as in Example \ref{2.14}. Then ${\rm\bf U}_1(\lam, j)$ is
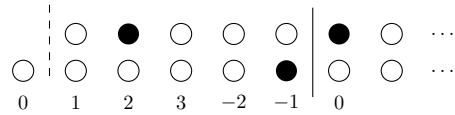
\[
\begin{tikzpicture}[scale=0.7, every node/.style={transform shape}]
\begin{scope}[yshift=0cm]
\node [] at (9,-1.3) {$0$};
\node [] at (10,-1.3) {$1$};
\node [] at (11,-1.3) {$2$};
\node [] at (12,-1.3) {$3$};
\node [] at (13,-1.3) {$-2$};
\node [] at (14,-1.3) {$-1$};
\node [] at (15,-1.3) {$0$};
\draw[black] (9,-0.7) circle (0.2cm);
\draw[black] (10,-0.7) circle (0.2cm);
\draw[black] (11,-0.7) circle (0.2cm);
\draw[black] (12,-0.7) circle (0.2cm);
\draw[black] (13,-0.7) circle (0.2cm);
\fill[black] (14,-0.7) circle (0.2cm);
\draw[black] (15,-0.7) circle (0.2cm);
\draw[black] (10,0) circle (0.2cm);
\fill[black](11,0) circle (0.2cm);
\draw[black] (12,0) circle (0.2cm);
\draw[black] (13,0) circle (0.2cm);
\draw[black] (14,0) circle (0.2cm);
\fill[black] (15,0) circle (0.2cm);
\draw[black] (16,0) circle (0.2cm);
\draw[black] (16,-0.7) circle (0.2cm);
\draw[dashed] (9.5,-0.8)--(9.5,0.6);
\draw[](14.5,-1.2)--node[]{}(14.5,0.5);
\node at (17,0) {$\dots$};
\node at (17,-0.7) {$\dots$};
\end{scope}
\end{tikzpicture}
\]
and ${\rm\bf U}_2(\lam, j)$ is
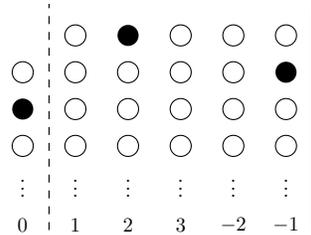
\[
\begin{tikzpicture}[scale=0.7, every node/.style={transform shape}]
\begin{scope}[yshift=0cm]
\node [] at (9,-3.6) {$0$};
\node [] at (10,-3.6) {$1$};
\node [] at (11,-3.6) {$2$};
\node [] at (12,-3.6) {$3$};
\node [] at (13,-3.6) {$-2$};
\node [] at (14,-3.6) {$-1$};
\draw[black] (9,-0.7) circle (0.2cm);
\draw[black] (10,-0.7) circle (0.2cm);
\draw[black] (11,-0.7) circle (0.2cm);
\draw[black] (12,-0.7) circle (0.2cm);
\draw[black] (13,-0.7) circle (0.2cm);
\fill[black] (14,-0.7) circle (0.2cm);
\draw[black] (10,0) circle (0.2cm);
\fill[black](11,0) circle (0.2cm);
\draw[black] (12,0) circle (0.2cm);
\draw[black] (13,0) circle (0.2cm);
\draw[black] (14,0) circle (0.2cm);
\draw[dashed] (9.5,-3.7)--(9.5,0.6);
\draw[](14.5,-3.7)--node[]{}(14.5,0.5);
\fill[black] (9,-1.4) circle (0.2cm);
\draw[black] (10,-1.4) circle (0.2cm);
\draw[black] (11,-1.4) circle (0.2cm);
\draw[black] (12,-1.4) circle (0.2cm);
\draw[black] (13,-1.4) circle (0.2cm);
\draw[black] (14,-1.4) circle (0.2cm);
\draw[black] (9,-2.1) circle (0.2cm);
\draw[black] (10,-2.1) circle (0.2cm);
\draw[black] (11,-2.1) circle (0.2cm);
\draw[black] (12,-2.1) circle (0.2cm);
\draw[black] (13,-2.1) circle (0.2cm);
\draw[black] (14,-2.1) circle (0.2cm);
\node at (9,-2.8) {$\vdots$};
\node at (10,-2.8) {$\vdots$};
\node at (11,-2.8) {$\vdots$};
\node at (12,-2.8) {$\vdots$};
\node at (13,-2.8) {$\vdots$};
\node at (14,-2.8) {$\vdots$};
\end{scope}
\end{tikzpicture}
\]
Finally, we get the image of $(\lam, j)$ under Uglov map
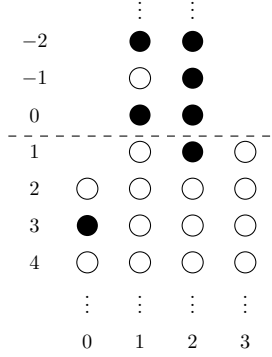
\[
\begin{tikzpicture}[scale=0.7, every node/.style={transform shape}]
\begin{scope}[yshift=0cm]
\node [] at (9,-3.6) {$0$};
\node [] at (10,-3.6) {$1$};
\node [] at (11,-3.6) {$2$};
\node [] at (12,-3.6) {$3$};
\node [] at (8,0.7) {$0$};
\node [] at (8,1.4) {$-1$};
\node [] at (8,2.1) {$-2$};
\node [] at (8,0) {$1$};
\node [] at (8,-0.7) {$2$};
\node [] at (8,-1.4) {$3$};
\node [] at (8,-2.1) {$4$};
\fill[black] (10,0.7) circle (0.2cm);
\fill[black] (11,0.7) circle (0.2cm);
\draw[black] (10,1.4) circle (0.2cm);
\fill[black] (11,1.4) circle (0.2cm);
\fill[black] (10,2.1) circle (0.2cm);
\fill[black] (11,2.1) circle (0.2cm);
\node at (10,2.8) {$\vdots$};
\node at (11,2.8) {$\vdots$};
\draw[black] (9,-0.7) circle (0.2cm);
\draw[black] (10,-0.7) circle (0.2cm);
\draw[black] (11,-0.7) circle (0.2cm);
\draw[black] (12,-0.7) circle (0.2cm);
\draw[black] (10,0) circle (0.2cm);
\fill[black](11,0) circle (0.2cm);
\draw[black] (12,0) circle (0.2cm);
\fill[black] (9,-1.4) circle (0.2cm);
\draw[black] (10,-1.4) circle (0.2cm);
\draw[black] (11,-1.4) circle (0.2cm);
\draw[black] (12,-1.4) circle (0.2cm);
\draw[black] (9,-2.1) circle (0.2cm);
\draw[black] (10,-2.1) circle (0.2cm);
\draw[black] (11,-2.1) circle (0.2cm);
\draw[black] (12,-2.1) circle (0.2cm);
\node at (9,-2.8) {$\vdots$};
\node at (10,-2.8) {$\vdots$};
\node at (11,-2.8) {$\vdots$};
\node at (12,-2.8) {$\vdots$};
\end{scope}
\end{tikzpicture}
\]}
\end{example}

\medskip


\subsection{Operations on abaci} In order to define a core abacus, we need to recall the operations on abaci.
We can understand the operations more easily via Uglov map.

\begin{definition}\label{2.17}
Let $(\lam, j)$ be an abacus. The elementary operations on ${\rm \bf U}(\lam, j)$ is defined as follows.
\begin{enumerate}
  \item [\rm(1)] First kind: If there is a bead at position $(x, y)$ and position $(x-1, y)$ is empty,
then move the bead to position $(x-1, y)$.
  \item [\rm(2)] Second kind: If the topmost position in a half column is occupied by a bead, then remove it.
\end{enumerate}
\end{definition}

Then the operations on an abacus have a simple description as below.

\begin{definition}\label{2.18}
An elementary operation on $(\lam, j)$ is a map $$(\lam, j)\rightarrow{\rm\bf U}^{-1}{(\tau(\rm\bf U}(\lam, j))),$$ 
where $\tau$ is an elementary operation on ${\rm \bf U}(\lam, j)$ defined in Definition \ref{2.17}.
\end{definition}

In order to give a visual description of elementary operations in Definition \ref{2.18}, we introduce the notion of {\em adjoint positions}.

\begin{definition}\label{2.19}
Let $(\lam, j)\in B_{\geq k}$. Two positions $x$ and $y$ are said to be $0$-adjoint 
if $x+y=-1$ and $0$ is not an $l$-index, or $x+y=-2$ and $0$ is an $l$-index.
Two positions are said to be $l$-adjoint if one of the following cases holds
\begin{enumerate}
\item [\rm (1)] $x+y=2l-1$, $k=-\infty$ or $0$ and $l+1$ is not an $l$-index;
\item [\rm (2)] $x+y=2l$, $k=-\infty$ or $0$ and $l+1$ is an $l$-index;
\item [\rm (3)] $x+y=4l-1$, $k=l$ and $0$ and $l+1$ are not $l$-indexes;
\item [\rm (4)] $x+y=4l+1$, $k=l$, $0$ is not an $l$-index and $l+1$ is an $l$-index;;
\item [\rm (5)] $x+y=4l+2$, $k=l+1$, and both $0$ and $l+1$ are $l$-indexes.
\end{enumerate}
\end{definition}

Then an elementary operation in Definition \ref{2.18} on an abacus $(\lam, j)$ of affine type $X$ belongs to one of the following four kinds.
\begin{enumerate}
  \item [\rm (1)] If two $0$-adjoint positions are empty, then set a bead on each position.
  \item [\rm (2)] If two $l$-adjoint positions are occupied by a bead, then remove them together.
  \item [\rm (3)] Move a bead at position $x$ to empty position $x-i$, 
  where $i\in \{l+1, 2l, 2l+1, 2l+2\}$, and determined by the affine type $X$.
  \item [\rm (4)] Remove or set a bead (corresponding to the elementary operations of second kind on ${\rm \bf U}(\lam, j)$).
\end{enumerate}

\begin{example}\label{2.20}{\rm
We can do on ${\rm\bf U}(\lam, j)$ in Example \ref{2.16} two first kind elementary operations and a second kind one. 
After doing these elementary operations, we get
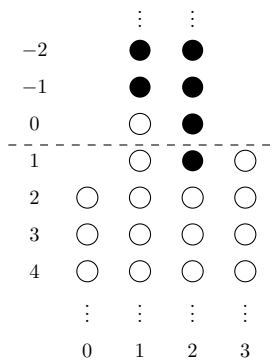
\[
\begin{tikzpicture}[scale=0.7, every node/.style={transform shape}]
\begin{scope}[yshift=0cm]
\node [] at (9,-3.6) {$0$};
\node [] at (10,-3.6) {$1$};
\node [] at (11,-3.6) {$2$};
\node [] at (12,-3.6) {$3$};
\node [] at (8,0.7) {$0$};
\node [] at (8,1.4) {$-1$};
\node [] at (8,2.1) {$-2$};
\node [] at (8,0) {$1$};
\node [] at (8,-0.7) {$2$};
\node [] at (8,-1.4) {$3$};
\node [] at (8,-2.1) {$4$};
\draw[black] (10,0.7) circle (0.2cm);
\fill[black] (11,0.7) circle (0.2cm);
\fill[black] (10,1.4) circle (0.2cm);
\fill[black] (11,1.4) circle (0.2cm);
\fill[black] (10,2.1) circle (0.2cm);
\fill[black] (11,2.1) circle (0.2cm);
\node at (10,2.8) {$\vdots$};
\node at (11,2.8) {$\vdots$};
\draw[black] (9,-0.7) circle (0.2cm);
\draw[black] (10,-0.7) circle (0.2cm);
\draw[black] (11,-0.7) circle (0.2cm);
\draw[black] (12,-0.7) circle (0.2cm);
\draw[black] (10,0) circle (0.2cm);
\fill[black](11,0) circle (0.2cm);
\draw[black] (12,0) circle (0.2cm);
\draw[black] (9,-1.4) circle (0.2cm);
\draw[black] (10,-1.4) circle (0.2cm);
\draw[black] (11,-1.4) circle (0.2cm);
\draw[black] (12,-1.4) circle (0.2cm);
\draw[black] (9,-2.1) circle (0.2cm);
\draw[black] (10,-2.1) circle (0.2cm);
\draw[black] (11,-2.1) circle (0.2cm);
\draw[black] (12,-2.1) circle (0.2cm);
\node at (9,-2.8) {$\vdots$};
\node at (10,-2.8) {$\vdots$};
\node at (11,-2.8) {$\vdots$};
\node at (12,-2.8) {$\vdots$};
\end{scope}
\end{tikzpicture}
\]

Correspondingly, the abacus $(\lam, j)$ becomes to
\[
\begin{tikzpicture}[scale=0.7, every node/.style={transform shape}]
\begin{scope}[yshift=0cm]
\node [] at (9,-0.6) {$0$};
\node [] at (8,-0.6) {$-1$};
\node [] at (7,-0.6) {$-2$};
\node [] at (6,-0.6) {$3$};
\node [] at (5,-0.6) {$2$};
\node [] at (4,-0.6) {$1$};
\node [] at (3,-0.6) {$0$};
\node [] at (10,-0.6) {$1$};
\node [] at (11,-0.6) {$2$};
\node [] at (12,-0.6) {$3$};
\node [] at (13,-0.6) {$-2$};
\node [] at (14,-0.6) {$-1$};
\node [] at (15,-0.6) {$0$};
\node [] at (16,-0.6) {$1$};
\node at (2,0) {$\dots$};
\node at (2,-0.5) {$\dots$};
\fill[black] (3.0,0) circle (0.2cm);
\fill[black](4,0) circle (0.2cm);
\fill[black] (5,0) circle (0.2cm);
\fill[black] (6,0) circle (0.2cm);
\fill[black] (7,0) circle (0.2cm);
\fill[black] (8,0) circle (0.2cm);
\fill[black] (9,0) circle (0.2cm);
\draw[black] (10,0) circle (0.2cm);
\fill[black](11,0) circle (0.2cm);
\draw[black] (12,0) circle (0.2cm);
\draw[black] (13,0) circle (0.2cm);
\fill[black] (14,0) circle (0.2cm);
\draw[black] (15,0) circle (0.2cm);
\draw[black] (16,0) circle (0.2cm);
\draw[dashed] (9.5,-0.8)--(9.5,0.6);
\node at (17,0) {$\dots$};
\node at (17,-0.5) {$\dots$};
\end{scope}
\end{tikzpicture}
\]
}
\end{example}

\bigskip


\section{Core abaci of classical affine types}

The goal of this section is to define and study $j$-core abaci of all classical affine types for arbitrary charge $j$.
It is important to point out that our notion of core abaci is a pair, not just a partition. 
This means that $(\lam, j)$ being a core abaci does not imply $(\lam, i)$ being a core too.
We will prove that the cores defined in \cite{HJ2} coincide with the $0$-cores of untwisted classical types in the sense of our definition.
We will also use the $j$-core abaci to parameterize affine Grassmanian.

\subsection{Core abaci of arbitrary charge}

Let us begin with a result from \cite{HLQ}.
\begin{lemma}\cite[Theorem 6.11, Lemma 6.12]{HLQ}\label{3.1}
Given a pair $(\lam, j)\in \Lambda_j-\beta$ of affine type $X$, the followings are equivalent.
\begin{enumerate}
\item[\rm(1)]\,${\rm def}(\lam, j)=0$.
\item[\rm(2)]\,The pair $(\lam, j)$ is in the same orbit as the abacus of weight $\Lambda_j$ under the $W$ action given in Definition \ref{2.7}.
\item[\rm(3)]\,No elementary operation can be done in the abacus $(\lam, j)$.
\end{enumerate}
\end{lemma}

Then we can give  the definition of {\em core abaci}. It is a natural generalization of the known concepts.
We will explain this in the third subsection.
\begin{definition}\label{3.2}
A pair $(\lam, j)$ of affine type $X$ satisfying one of the conditions in Lemma \ref{3.1} is called a $j$-core abacus of type $X$.
\end{definition}

Note that by \cite[Proposition 11.4]{Kac}, if a core abacus $(\lam, j)\in \Lambda_j-\beta$, then no other abacus belong to $\Lambda_j-\beta$. 
Based on this fact, we sometimes call $\Lambda_j-\beta$ a core abacus for convenience. We also write ${\rm ht}(\lam, j)$ to mean ${\rm ht}(\beta)$.

\smallskip

Given a core abacus $(\lam, j)$ of type $X$ which corresponds to $\Lambda_j-\beta$, we have $|\lam|={\rm ht}(\beta)$ 
if $X=C_l^{(1)}, D_{l+1}^{(2)}$ or $A_{2l}^{(2)}$. However, if $X=B_l^{(1)}$ or $A_{2l-1}^{(2)}$  and the coefficient of 
$\alpha_0$ in $\beta$ is not zero, and $X=D_l^{(1)}$  and the coefficient of 
$\alpha_0$ or $\alpha_l$ in $\beta$ is not zero, then $|\lam|\neq{\rm ht}(\beta)$.
This is the reason why we establish Diophantine equations by using the height formula in Section 4.

\begin{example}{\rm
Let $(\lam, 0)$ be a core abacus of type $D_5^{(1)}$, where $\lam=(11, 8, 8, 5, 4)$ (see Example \ref{2.5} for the abacus display) and so  $|\lam|=36$. 
Then the corresponding root $\beta=4\alpha_0+2\alpha_1+7\alpha_2+8\alpha_3+3\alpha_4+4\alpha_5$.
Therefore, ${\rm ht}(\beta)=28$. It is not equal to $|\lam|$.}
\end{example}

Given an abacus $(\lam, j)$, denote by ${\rm\bf s}({\rm\bf U}(\lam, j))$ the vector obtained 
by treating each runner of ${\rm\bf U}(\lam, j)$ as an abacus and then taking charges in order.
Then we can define the so-called Uglov vector, which will be used to compute the height of $(\lam, j)$.

\begin{definition}\label{3.3}
Given a core abacus $(\lam, j)\in B_{\geq k}$ of affine type $X$, define the Uglov vector of $(\lam, j)$ to be
$$u(\lam, j)=\begin{cases}
        {\rm\bf s}({\rm\bf U}(\lam, j))-\frac{3}{2}{\bf 1}, & \mbox{\rm if } k=l\, {\rm or}\, l+1; \\
        {\rm\bf s}({\rm\bf U}(\lam, j))-{\bf 1}, & \mbox{\rm otherwise}.
      \end{cases}$$
\end{definition}

\begin{example}\label{3.4}{\rm
Let $(\lam, j)=((4,2,1,1,1,1,1), 1)$ be an abacus of affine type $D^{(2)}_3$. The following is the abacus display.
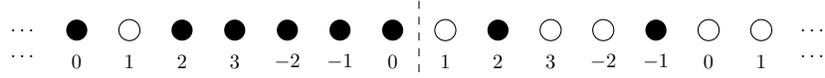
\[
\begin{tikzpicture}[scale=0.7, every node/.style={transform shape}]
\begin{scope}[yshift=0cm]
\node [] at (9,-0.6) {$0$};
\node [] at (8,-0.6) {$-1$};
\node [] at (7,-0.6) {$-2$};
\node [] at (6,-0.6) {$3$};
\node [] at (5,-0.6) {$2$};
\node [] at (4,-0.6) {$1$};
\node [] at (3,-0.6) {$0$};
\node [] at (10,-0.6) {$1$};
\node [] at (11,-0.6) {$2$};
\node [] at (12,-0.6) {$3$};
\node [] at (13,-0.6) {$-2$};
\node [] at (14,-0.6) {$-1$};
\node [] at (15,-0.6) {$0$};
\node [] at (16,-0.6) {$1$};
\node at (2,0) {$\dots$};
\node at (2,-0.5) {$\dots$};
\fill[black] (3.0,0) circle (0.2cm);
\draw[black](4,0) circle (0.2cm);
\fill[black] (5,0) circle (0.2cm);
\fill[black] (6,0) circle (0.2cm);
\fill[black] (7,0) circle (0.2cm);
\fill[black] (8,0) circle (0.2cm);
\fill[black] (9,0) circle (0.2cm);
\draw[black] (10,0) circle (0.2cm);
\fill[black](11,0) circle (0.2cm);
\draw[black] (12,0) circle (0.2cm);
\draw[black] (13,0) circle (0.2cm);
\fill[black] (14,0) circle (0.2cm);
\draw[black] (15,0) circle (0.2cm);
\draw[black] (16,0) circle (0.2cm);
\draw[dashed] (9.5,-0.8)--(9.5,0.6);
\node at (17,0) {$\dots$};
\node at (17,-0.5) {$\dots$};
\end{scope}
\end{tikzpicture}
\]
Then ${\rm\bf U}(\lam, 1)$ is 
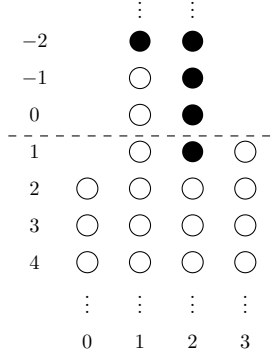
\[
\begin{tikzpicture}[scale=0.7, every node/.style={transform shape}]
\begin{scope}[yshift=0cm]
\node [] at (9,-3.6) {$0$};
\node [] at (10,-3.6) {$1$};
\node [] at (11,-3.6) {$2$};
\node [] at (12,-3.6) {$3$};
\node [] at (8,0.7) {$0$};
\node [] at (8,1.4) {$-1$};
\node [] at (8,2.1) {$-2$};
\node [] at (8,0) {$1$};
\node [] at (8,-0.7) {$2$};
\node [] at (8,-1.4) {$3$};
\node [] at (8,-2.1) {$4$};
\draw[black] (10,0.7) circle (0.2cm);
\fill[black] (11,0.7) circle (0.2cm);
\draw[black] (10,1.4) circle (0.2cm);
\fill[black] (11,1.4) circle (0.2cm);
\fill[black] (10,2.1) circle (0.2cm);
\fill[black] (11,2.1) circle (0.2cm);
\node at (10,2.8) {$\vdots$};
\node at (11,2.8) {$\vdots$};
\draw[black] (9,-0.7) circle (0.2cm);
\draw[black] (10,-0.7) circle (0.2cm);
\draw[black] (11,-0.7) circle (0.2cm);
\draw[black] (12,-0.7) circle (0.2cm);
\draw[black] (10,0) circle (0.2cm);
\fill[black](11,0) circle (0.2cm);
\draw[black] (12,0) circle (0.2cm);
\draw[black] (9,-1.4) circle (0.2cm);
\draw[black] (10,-1.4) circle (0.2cm);
\draw[black] (11,-1.4) circle (0.2cm);
\draw[black] (12,-1.4) circle (0.2cm);
\draw[black] (9,-2.1) circle (0.2cm);
\draw[black] (10,-2.1) circle (0.2cm);
\draw[black] (11,-2.1) circle (0.2cm);
\draw[black] (12,-2.1) circle (0.2cm);
\node at (9,-2.8) {$\vdots$};
\node at (10,-2.8) {$\vdots$};
\node at (11,-2.8) {$\vdots$};
\node at (12,-2.8) {$\vdots$};
\end{scope}
\end{tikzpicture}
\]
and $u(\lam, 1)=(-2, 1)$.}
\end{example}

We now study the properties of Uglov vectors of core abaci. The following two lemmas play an important role in the whole paper.

\begin{lemma}\label{3.5}
Let $(\lam, j)$ be a core abacus (not half). Then the number of odd components in the Uglov vector $u$ is $j$.
\end{lemma}

\begin{proof}
According to the rule of Uglov map, every odd component in $u$ is arising from a pair of beads 
at positions $(i, k)$ and $(-i, -k-1)$ in the abacus $(\lam, j)$, and the number of such pairs is just the charge of $(\lam, j)$.
\end{proof}

Next we concern half core abaci. Here is a direct corollary of Uglov map. It can be checked case by case and we omit the details.

\begin{lemma}\label{3.6}
Let $(\lam, j)$ be a half core abacus. If the Uglov vector of $(\lam, j)$ is $u$, 
then the Uglov vector of $(\ddot{\lam}, j)$ is $2u$.
\end{lemma}

\smallskip

We conclude this subsection by considering the conjugation of core abaci.
\begin{definition}\label{3.7}
Let $\lam$ be a partition. Then the conjugation of the abacus $(\lam, j)$ is defined to be $(\lam', l-j)$. 
\end{definition}

Let us give a simple result for later use.
\begin{lemma}\label{3.8}
Let $(\lambda, j)$ be a core abacus of type $C_l^{(1)}$, $D_l^{(1)}$ or $D_{l+1}^{(2)}$ with $j$ satisfies 
\begin{equation*}
\begin{cases}
0\leq j\leq l, & \mbox{\rm if} \,\,(\lam, j)\,\, \text{\rm is of type}\,\, C_l^{(1)}; \\
1\leq j\leq l-1, & \mbox{\rm if } (\lam, j) \,\, \text{\rm is of type} \,\,D_{l+1}^{(2)};  \\
2\leq j\leq l-2, & \mbox{\rm if}\,\, (\lam, j) \text{\rm is of type}\,\, D_l^{(1)}.
\end{cases}\eqno(3.2.1)
\end{equation*}
Then there is a bijection between the actions of $f_i$ on $(\lam, j)$ and that of $f_{l-i}$ on $(\lam', l-j)$ for each $i=0,1,\dots,l$.
\end{lemma}

Based on Lemma \ref{3.8}, we can prove that the conjugation of a core abacus is a core abacus too for certain affine types.

\begin{lemma}\label{3.9}
Let $(\lam, j)$ be a core abacus of type $C_l^{(1)}$, $D_l^{(1)}$ or $D_{l+1}^{(2)}$ with $j$ satisfies {\rm (3.2.1)}.
Then $(\lam', l-s)$ is also a core abacus.
\end{lemma}

\begin{proof}
Let $(\lam, j)\in \Lambda_j - \beta$ and $\beta=\sum_{i=0}^{l} k_i\alpha_i$.
It follows from Lemma \ref{3.8} that $(\lam', l-j)\in \Lambda_{l-j} -\beta'$, where $\beta'=\sum_{i=0}^{l} k_{l-i}\alpha_i$. 
As a result, $(\Lambda_j, \beta)=k_j=(\Lambda_{l-j}, \beta')$
and $(\beta, \beta)=(\beta', \beta')$. That is, ${\rm def}(\Lambda_{l-j} - \beta')={\rm def}(\Lambda_j - \beta)=0$, or $(\lam', l-s)$ is also a core.
\end{proof}

It is helpful to point out that Lemma \ref{3.9} does not hold for other types. Let us give an example.

\begin{example}\label{3.10}
{\rm
We consider core abaci of type $B_5^{(1)}$. Take $\lam=(5, 1, 1)$. Clearly $(\lam, 2)$ is a core abacus 
and the conjugation $(\lam', 3)$ is not a core. Interestingly, $(\lam', 4)$ is a core abacus. Moreover,
take $\lam=(9,1)$. Then $(\lam, 2)$ is a core abacus. Surprisingly, $(\lam', i)$ is not a core abacus for any $1< i< 5$.
However, Lemma \ref{3.9} holds for certain core abacus. For instance, $((5, 1), 2)$.}
\end{example}

\smallskip


\subsection{Parametrization of affine Grassmannian}
The aim of this subsection is to parameterize the affine Grassmannian $W^j$ 
by the collection of $j$-core abaci for arbitrary classical affine types.
Let us realize an affine Weyl groups by Coxeter complex first. The main reference is \cite{H}.

Let $V$ be an euclidean space with $\dim V=l$ and $\Phi$ a root system. For each root $\alpha\in \Phi$ and $k\in \mathbb{Z}$, 
define the associated {\em affine hyperplane} 
$$H_{\alpha, k}:=\left\{v\in V \mid (v,\alpha)=k \right\}.$$
The corresponding affine reflection is $\sigma_{\alpha, k}(v):=v-((v,\alpha)-k)\alpha^{\vee}$.
Write $\sigma_{\alpha_i, 0}$ as $\sigma_i$ $(i\neq 0)$ and write $\sigma_{\theta, 1}$ as $\sigma_0$, where $\theta$ is the highest root. 
Then $S=\{\sigma_0, \sigma_1, \dots\}$ generate an affine Weyl group.
Denote by $I_i=S-\sigma_i$. For $\sigma_j\in I_0$, define $$A_{j}^+=\{v\in V\mid (v, \alpha_j)>0\}$$ and
$$A_{0}^+=\{v\in V\mid (v, \theta)<1\}.$$
Take the {\em fundamental alcove} $\mathcal{A}_e$ to be $\bigcap_{j} A_{j}^+$, and then for each $w\in W$ 
the corresponding alcove is $\mathcal{A}_w:=w\mathcal{A}_e$. It is well known that $w\longrightarrow \mathcal{A}_w$
is a bijection from the affine Weyl group $W$ to the collection of alcoves.

The {\em right affine Grassmannian} is $$W^i:=\{w\in W\mid \ell(w\sigma)>\ell(w)\,\, \text{for each}\,\, \sigma\in I_i\},$$ 
in which the elements having no reduced expression ends with a letter not being $i$. Similarly, we can define the
{\em left affine Grassmannian} $$^i W:=\{w\in W\mid \ell(\sigma w)>\ell(w)\,\, \text{for each}\,\, \sigma\in I_i\}.$$
Note that $W^i=(^i W)^{-1}$.

Based on the definition of generalized Tits cone $\mathcal{C}_i:=\bigcap_{j\neq i} A_{j}^+$, one can get a subset of $W$: 
$$\mathcal{C}_i(W):=\{w\in W\mid \mathcal{A}_w\subset \mathcal{C}_i\}.$$ In \cite{H}, Humphreys proved a result about affine Grassmannian as follows.

\begin{lemma}\cite[Lemma 5.13]{H}\label{3.11}
Keep notations as above. Then $^i W=\mathcal{C}_i(W)$.
\end{lemma}

By the definition of core abaci, each $w\in W^j$ clearly determines a core abacus $w(\varnothing, j)$. 
On the other hand, we have the following result.

\begin{lemma}\label{3.12}
Let $(\lam, j)$ be a core abacus. Then there exists a unique element $w\in W^j$ such that $w(\varnothing, j)=(\lam, j)$. 
\end{lemma}

\begin{proof}
According to Definition \ref{3.2}, $(\lam, j)$ and $(\varnothing, j)$ are in the same orbit denoted by ${\rm O}_{\varnothing, j}$ under the action of $W$.
Note that there is a bijection between ${\rm O}_{\varnothing, j}$ and the left coset of the stabilizer subgroup of $(\varnothing, j)$.
Clearly, the stabilizer of $(\varnothing, j)$ is the subgroup $W_j$ of $W$, which is generated by $\{\sigma_i\mid i\neq j\}$.
Then by \cite[Corollary 2.4.5]{BB}, each left coset  $gW_j$ has a unique representative $w$ of minimal length and the 
system of such minimal coset representative is just $W^j$. 
\end{proof}

The element in $W$ corresponding to the core abacus $(\lam, j)$ will be denoted by $w_{\lam, j}$.
We have established a bijection from the collection of $j$-core abaci to the affine Grassmannian $W^j$ now.
This implies that we can fill the alcoves $\mathcal{A}_w$ in the generalized Tits cone $\mathcal{C}_j$ by $j$-cores $w(\varnothing, j)$.
By abusing notations, we also denote the set of $j$-cores by $\mathcal{C}_j$.

Note that each core abacus (not half ones) determines a unique partition whenever the charge is fixed.
Then we can fill the alcoves in the generalized Tits cone $\mathcal{C}_j$ by partitions recursively as follows:
\begin{enumerate}
  \item [\rm (1)] \,Put $\varnothing$ inside the fundamental alcove.
  \item [\rm (2)] \,Fill in all other alcoves in the Tits cone $\mathcal{C}_j$ partitions recursively by adding all possible
$i$-nodes each time a new $i$-wall is crossed.
\end{enumerate}
This is a natural generalization of the charge 0 cores that are already known (for example, see \cite[Fig 5]{BCG}).
Let us give an example.
\begin{example}\label{3.13}
{\rm
The following figure gives a bijection between 1-core abaci ({\bf here 1 means charge 1}) 
and alcoves in the generalized Tits cone $\mathcal{C}_1$ of affine type $C_2^{(1)}.$}
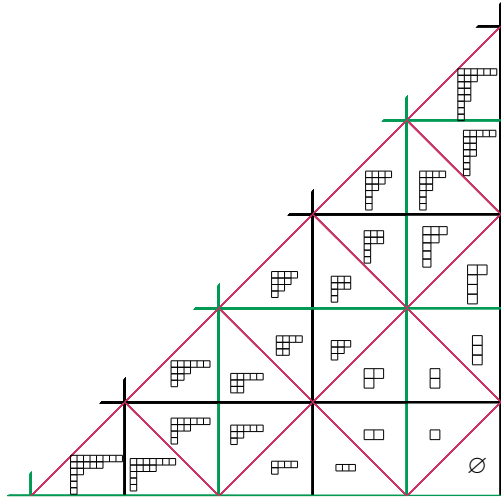
\begin{figure}[h!]
\[
	\begin{tikzpicture}[scale=1.25]
		\clip (-5.25,0) -- (0,0) -- (0,5.25) -- cycle;
		
		\foreach \i in {-5,...,0}
		{
			\foreach \j in {0,...,5}
			{
				
			\draw[black,thick] (0,\j) -- + (90:5);
			\draw[black,thick] (-2,\j) -- + (90:5);
			\draw[black,thick] (-4,\j) -- + (90:5);
			\draw[ForestGreen, thick] (-1,\j) -- + (90:5);
			\draw[ForestGreen, thick] (-3,\j) -- + (90:5);
			\draw[ForestGreen, thick] (-5,\j) -- + (90:5);
			
			\draw[ForestGreen, thick] (\i,0) -- + (180:5);
			\draw[ForestGreen, thick] (\i,2) -- + (180:5);
			\draw[ForestGreen, thick] (\i,4) -- + (180:5);
			\draw[black,thick] (\i,1) -- + (180:5);
			\draw[black,thick] (\i,3) -- + (180:5);
			\draw[black,thick] (\i,5) -- + (180:5);
			
			\draw[purple!80] (-1,0) -- + (45:{sqrt(2)});
			\draw[purple!80] (-3,0) -- + (45:{3*sqrt(2)});
			\draw[purple!80] (-5,0) -- + (45:{5*sqrt(2)});
			\draw[purple!80] (-3,2) -- + (-45:{2*sqrt(2)});
			\draw[purple!80] (-4,1) -- + (-45:{sqrt(2)});
			\draw[purple!80] (-2,3) -- + (-45:{2*sqrt(2)});
			\draw[purple!80] (-1,4) -- + (-45:{sqrt(2)});
			
			}
		}
		\node[anchor = mid, scale=0.9] at ( -0.25, 0.3 ) {$\varnothing$};
		\node[anchor = mid, scale=0.3] at ( -0.7, 0.65 ) {$\yng(1)$};
		\node[anchor = mid, scale=0.3] at ( -1.35, 0.65) {$\yng(2)$};
		\node[anchor = mid, scale=0.3] at ( -0.7, 1.25 ) {$\yng(1,1)$};
		\node[anchor = mid, scale=0.3] at ( -1.35, 1.25 ) {$\yng(2,1)$};
		\node[anchor = mid, scale=0.3] at ( -0.25, 1.55 ) {$\yng(1,1,1)$};
		\node[anchor = mid, scale=0.3] at ( -0.25, 2.25 ) {$\yng(2,1,1,1)$};
		\node[anchor = mid, scale=0.25] at ( -0.7, 2.65 ) {$\yng(3,2,1,1,1)$};
		\node[anchor = mid, scale=0.2] at ( -1.65, 0.3 ) {$\yng(3)$};
		\node[anchor = mid, scale=0.2] at ( -1.7, 1.55 ) {$\yng(3,2,1)$};
		\node[anchor = mid, scale=0.2] at ( -1.7, 2.2 ) {$\yng(3,3,1,1)$};
		\node[anchor = mid, scale=0.2] at ( -2.3, 2.25 ) {$\yng(4,3,2,1)$};
		\node[anchor = mid, scale=0.2] at ( -2.25, 1.6 ) {$\yng(4,2,2)$};
		\node[anchor = mid, scale=0.2] at ( -2.3, 0.3 ) {$\yng(4,1)$};
		\node[anchor = mid, scale=0.2] at ( -2.7, 0.65 ) {$\yng(5,2,1)$};
		\node[anchor = mid, scale=0.2] at ( -2.7, 1.2 ) {$\yng(5,2,2)$};
		\node[anchor = mid, scale=0.2] at ( -1.35, 2.65 ) {$\yng(3,3,1,1,1)$};
		\node[anchor = mid, scale=0.2] at ( -0.725, 3.25 ) {$\yng(4,2,2,1,1,1)$};
		\node[anchor = mid, scale=0.2] at ( -0.225, 3.65 ) {$\yng(5,2,2,2,1,1,1)$};
		\node[anchor = mid, scale=0.2] at ( -0.25, 4.27 ) {$\yng(6,3,2,2,2,1,1,1)$};
		\node[anchor = mid, scale=0.2] at ( -3.3, 0.69 ) {$\yng(6,3,1,1)$};
		\node[anchor = mid, scale=0.2] at ( -3.70, 0.225 ) {$\yng(7,4,1,1,1)$};
		\node[anchor = mid, scale=0.2] at ( -4.3, 0.225 ) {$\yng(8,5,2,1,1,1)$};
		\node[anchor = mid, scale=0.2] at ( -3.3, 1.3 ) {$\yng(6,3,2,1)$};
		\node[anchor = mid, scale=0.2] at ( -1.3, 3.25 ) {$\yng(4,3,2,1,1,1)$};		
\end{tikzpicture}
\]
\end{figure}
\end{example}

\smallskip


\subsection{Comparison with existing concepts}

Next we study the relationship between the definition of cores appeared in \cite{BCG, HJ2} with our's. 
We only need to consider types $\tilde{C}/C$ and $\tilde{D}/D$ by \cite[Table 4]{HJ2}.
It is helpful to point out that all of them are certain special partitions. 
This implies that all the existing concepts are of charge 0 cases. 
We will also describe the 0-core abaci of each classical affine type by cores of type $A^{(1)}$.

\subsubsection{Type $C_l^{(1)}$} We will prove that $(\lam, 0)$ is a core abacus of type $C_l^{(1)}$ 
if and only if $(\lam, 0)$ is a self-conjugate $0$-core of type $A_{2l-1}^{(1)}$, which coincides with the $\tilde{C}/C$ cores defined in \cite{HJ2}. 
We first prove that a core of type $C_l^{(1)}$ is a core of type $A_{2l-1}^{(1)}$.
\begin{lemma}\label{3.14}
Let $(\lam, j)$ be a core abacus of type $C_l^{(1)}$. Then
\begin{enumerate}
\item[{\rm (1)}] if position $i$ has a bead, then position $i-2l$ has a bead too; 
\item[{\rm (2)}] if position $i$ is empty, then position $i+2l$ is empty too. 
\end{enumerate}
\end{lemma}

\begin{proof}
(1) Since $(\lam, j)$ is a core abacus of type $C_l^{(1)}$, by Definition \ref{3.2}, none elementary operations can be done on $(\lam, j)$. 
Consequently, if position $i$ has a bead, then position $2l-1-i$ is empty. By the same reason, position $i-2l$ has a bead.

Part (2) can be proved similarly.
\end{proof}
\begin{corollary}\label{3.15}
Let $(\lam, j)$ be a core abacus of type $C_l^{(1)}$. Then 
\begin{enumerate}
\item[{\rm (1)}] if position $i$ has bead, then there does not exists $j<i$ such that positions $j, j-1, \dots, j+1-2l$ are all empty;
\item[{\rm (2)}] if position $i$ is empty, then there does not exists $j>i$ such that all positions $j, j+1, \dots, j-1+2l$ have a bead.
\end{enumerate}
\end{corollary}

\begin{proof}
(1) Clearly, there exists some $k\in \mathbb{Z}$ such that $i-2kl\in\{j, j-1, \dots, j+1-2l\}$. By Lemma \ref{3.14}, position $i-2kl$ has a bead.

(2) It is proved similarly as (1).
\end{proof}

\begin{lemma}\label{3.16}
Each $0$-core abacus of type $C_l^{(1)}$ is self-conjugate.
\end{lemma}

\begin{proof}
Suppose that $(\lambda, 0)$ is a core abacus of type $C_l^{(1)}$. If in $(\lambda, 0)$ the $j$-th position is empty, 
then the $(-1-j)$-th position has a bead by Definition \ref{3.2}. Consequently, the $j$-th position in $(\lambda', 0)$ is empty.
	
If the $j$-th position in $(\lambda, 0)$ has a bead, then we have from the charge being 0 that the $(-1-j)$-th position is empty. 
This implies that the $j$-th position in $(\lambda', 0)$ has a bead.
	
In a word, we have proved that the $j$-th position in $(\lambda, 0)$ is empty if and only if the $j$-th position in $(\lambda', 0)$ is empty. That is, $\lambda = \lambda'$.
\end{proof}

We have realized a half of our aim proposed at the beginning, and now we consider the other half.

\begin{lemma}\label{3.17}
Let $\lam$ be a self-conjugate $0$-core of type $A_{2l-1}^{(1)}$. Then $(\lam, 0)$ is a core abacus of type $C_l^{(1)}$.
\end{lemma}

\begin{proof}
Let $\lam$ be a self-conjugate $0$-core of type $A_{2l-1}^{(1)}$. We give a proof by contradiction. 
If there exist two positions $i$ and $-1-i$ being empty in $(\lam, 0)$, then both positions $i$ and $-1-i$ in $(\lam', 0)$ have a bead.
It is impossible because $\lam$ is self-conjugate. On the other hand,  if there exist two positions $i$ and $2l-1-i$ being occupied by a bead, 
then position $-i-1$ has a bead because $\lam$ is a $0$-core of type $A_{2l-1}^{(1)}$. However, $\lam$ being self-conjugate forces 
the position $-i-1$ to be empty. It is a contradiction and we complete the proof.
\end{proof}

Combining Lemmas \ref{3.16} and \ref{3.17}, we get the following result.

\begin{lemma}\label{3.18}
The pair $(\lam, 0)$ is a core abacus of type $C_l^{(1)}$ if and only if it is a self-conjugate core abacus of type $A_{2l-1}^{(1)}$.
\end{lemma}

\subsubsection{Type $D^{(1)}_l$}
We shall connect the cores of $\tilde{D}/D$ type to $0$-core abaci of affine type $D_l^{(1)}$ based on the preparation given in Section 2. 
\begin{definition}\cite[Definition 5.1]{HJ2}\label{3.19}
A partition $\lam$ is said to be even if the number of the diagonal nodes in $[\lam]$ is even.
\end{definition}

It is easy to check the following simple result.
\begin{lemma}\label{3.20}
A partition $\lam$ is even if and only if the number of beads on the right side of the dashed line in $(\lam, 0)$ is even.
\end{lemma}

\begin{lemma}\label{3.21}
Let $(\lam, 0)$ be an abacus of type $D_l^{(1)}$. Then it is a core abacus if and only if 
$(\ddot{\lam}, 2\Lambda_0)$ is an even self-conjugate $0$-core of type $A^{(1)}_{2l-1}$,
or a core of $\tilde{D}/D$ type in the sense of \cite{HJ2}.  
\end{lemma}

\begin{proof}
$``\Longrightarrow"$ It is easy to know that $(\ddot{\lam}, 2\Lambda_0)$ is self-conjugate with charge 0 by its construction. 
On the other hand, since $(\lam, 0)$ is an abacus of type $D_l^{(1)}$, all beads are generated from the action of $f_0$. 
Note that each $f_0$-action gives two beads and consequently, the number of beads in $(\lam, 0)$ is even. Therefore, the
number of beads on the right side of the dashed line in $(\ddot{\lam}, 2\Lambda_0)$ is even. That is, $\ddot{\lam}$ is an even partition by Lemma \ref{3.20}. 
Finally, it is not difficult to check that $(\ddot{\lam}, 2\Lambda_0)$ is a $0$-core abacus of type $A^{(1)}_{2l-1}$ by the fact $(\lam, 0)$ being a core.

$``\Longleftarrow"$ Because $(\ddot{\lam}, 2\Lambda_0)$ is an even self-conjugate $0$-core of type $A^{(1)}_{2l-1}$, by Proposition \ref{3.18}, 
$(\ddot{\lam}, 2\Lambda_0)$ is 0-core abacus of type $C_l^{(1)}$.
Consequently, none elementary operations of type $C_l^{(1)}$ and therefore type $D_l^{(1)}$, can be done on $(\lam, 0)$ by Definition \ref{2.18}. 
Then by Definition \ref{3.2}, $(\lam, 0)$ is a core abacus of type $D_l^{(1)}$.
\end{proof}

\begin{remark}{\rm 
Given a core $\mu$ of type $\tilde{D}/D$, then it is easy to check that 
in $(\mu, 0)$ the half on the right side of the dashed line is a core abacus of type $D_l^{(1)}$.
Combining this fact with Lemma \ref{3.21}, we get a bijection between $0$-core abaci of type $D_l^{(1)}$ and cores of type $\tilde{D}/D$.}
\end{remark}

\subsubsection{Type $A_{2l-1}^{(2)}$} 
The following lemma is simple by the definition of core abaci of affine type $A_{2l-1}^{(2)}$.
\begin{lemma}\label{3.23}
Let  $\lam$ be a partition and $0\leq j\leq l$. Let $(\lambda, j)$ be an abacus of type $A_{2l-1}^{(2)}$ with ${\rm def}(\lambda, j)=0$. Then we have
\begin{enumerate}
  \item [\rm (1)]\, if position $i$ has a bead, then position $i-2l$ has a bead too;
  \item [\rm (2)]\, if position $i$ is empty, then position $i+2l$ is empty too.
\end{enumerate}	
\end{lemma}

\begin{lemma}\label{3.24}
Let $(\lambda, 0)$ be an abacus of type $A_{2l-1}^{(2)}$. Then $(\lambda, 0)$ is a core abacus if and only if 
$(\ddot{\lambda}, 2\Lambda_0)$ is an even self-conjugate $0$-core of type $A_{2l-1}^{(1)}$.
\end{lemma}

\begin{proof}
``$\Longrightarrow$" According to Definition \ref{3.2}, it is easy to check that $(\ddot{\lambda},2\Lambda_0)$ is self-conjugate with charge $0$. 
Note that all beads are generated from the action of $f_0$ and each $f_0$-action gives two beads. Therefore, the number of beads in $(\lambda, 0)$ is even. 
That is, the number of beads to the right of dashed line in $(\ddot{\lambda},2\Lambda_0)$ is even. 
Moreover, since $\rm def(\lambda, 0)=0$, by the construction of $(\ddot{\lambda},2\Lambda_0)$ and
Lemma \ref{3.23}, $(\ddot{\lambda},2\Lambda_0)$ is a $0$-core of type $A_{2l-1}^{(1)}$.

\smallskip
	
``$\Longleftarrow$" Since $(\ddot{\lambda}, 2\Lambda_0)$ is even, $(\lam, 0)$ can be viewed as an abacus of affine type $A_{2l-1}^{(2)}$.
Clearly,  $(\ddot{\lambda}, 2\Lambda_0)$ is a core of affine type $C_l^{(1)}$ by Lemma \ref{3.18}. 
This implies none elementary operation of affine type $C_l^{(1)}$ can be done in $(\ddot{\lambda}, 2\Lambda_0)$, 
and consequently none elementary operation of type $A^{(2)}_{2l-1}$ can be done in $(\lam, 0)$. 
That is, it is a core abacus.
\end{proof}

The following result is an easy corollary of Lemma \ref{3.23}. We omit the proof here.

\begin{lemma}\label{3.24'}
Let $(\lambda, l)$ be an abacus of type $A_{2l-1}^{(2)}$. Then $(\lambda, l)$ is a core abacus if and only if 
it is a self-conjugate $0$-core of type $A_{2l-1}^{(1)}$.
\end{lemma}

\subsubsection {Types $A_{2l}^{(2)}$, $B_{l}^{(1)}$ $D_{l+1}^{(2)}$} The core abacus of these three types can be studied similarly. We summarize a lemma without proof.

\begin{lemma} \label{3.25}
Given an abacus $(\lam, 0)$, then
\begin{enumerate}
\item[{\rm (1)}] $(\lambda, 0)$ is a core abacus of type $A_{2l}^{(2)}$ 
if and only if $(\ddot{\lambda},2\Lambda_0)$ is a 0-core of type $A_{2l}^{(1)}$.
\item[{\rm (2)}] $(\lambda, 0)$ is a core abacus of type $B_{l}^{(1)}$ 
if and only if $(\ddot{\lambda},2\Lambda_0)$ is an even self-conjugate $0$-core of type $A_{2l}^{(1)}$;
\item[{\rm (3)}] $(\lambda, 0)$ is a core abacus of type $D_{l+1}^{(2)}$ 
if and only if $ (\ddot{\lambda},2\Lambda_0)$ is a $0$-core of type $A_{2l+1}^{(1)}$.
\end{enumerate}
\end{lemma}

Note that Lecouvey and Wahiche investigated in \cite{LW} the problem of 
parameterizing the affine Grassmannian elements by self-conjugate cores of type $A_{2l-1}^{(1)}$.
Combining Lemmas \ref{3.18}, \ref{3.21} and \ref{3.25} (2) yields the following proposition.

\begin{proposition}
Affine Grassmannians $W^0$ of all classical types except $A^{(1)}$ 
can be parameterized by {\rm ({\bf certain})} self-conjugate 0-cores of type $A_{2l-1}^{(1)}$ or $A_{2l}^{(1)}$.
\end{proposition}

\smallskip


\subsection{Semidirect product and Uglov map}

It is known that for a core abacus the Uglov map is compatible with the semidirect product decomposition in the case of charge zero.
We generalize in this subsection the result to core abaci of arbitrary charge, 
that is, we shall prove that the Uglov map is compatible with the adjusted semidirect product decomposition, 
where the adjustment is determined by the charge.

In order to study the semidirect product structure and give the height formula in the following section, 
we first study the action of an affine Weyl group on the Uglov vector of a core abaci. 
The following result can be obtained by routine analysis on Uglov map. We omit the proof here.

\begin{lemma}\label{3.26}
Let $(\lam, j)$ be an abacus of type $X$ with the associated affine Weyl group $W$.
Assume that the Uglov vector of $(\lam, j)$ is $u(\lam, j)=(u_1, \dots, u_l)$. Then
\begin{enumerate}
\item[\rm(1)] $u(\sigma_i(\lam, j))=(u_1, \dots, u_{i+1}, u_i, \dots, u_l), \quad \text{\rm if} \quad 1\leq i\leq l-1$;
\smallskip
\item[\rm(2)] $u(\sigma_l(\lam, j))=\begin{cases}
(u_1, \dots, u_{l-1}, -u_l), &\text{\rm if $X\neq D_l^{(1)}$};\\
(u_1, \dots, -u_{l}, -u_{l-1}), &\text{\rm if $X= D_l^{(1)}$}.
\end{cases}$
\smallskip
\item[\rm(3)] $u(\sigma_0(\lam, j))=\begin{cases}
\smallskip
(\frac{a_j^{\vee}}{a_0^{\vee}}-u_2, \frac{a_j^{\vee}}{a_0^{\vee}}-u_1, \dots, u_l), &\text{\rm if $X= B_l^{(1)}, A_{2l-1}^{(2)}, D_l^{(1)}$};\\
\smallskip
(\frac{a_j^{\vee}}{a_0^{\vee}}-u_1, \dots, u_{l}), &\text{\rm if $X= D_{l+1}^{(2)}, A_{2l}^{(2)}$};\\
(\frac{2a_j^{\vee}}{a_0^{\vee}}-u_1, \dots, u_{l}), & \text{\rm if $X= C_l^{(1)}$.}
\end{cases}$
\end{enumerate}
\end{lemma}

For the uniformity of description of certain results in the following, we introduce the notion of weighted Uglov vectors.

\begin{definition}
Given a core abacus $(\lam, j)$ of affine type $X$ with Uglov vector $u$, define the weighted Uglov vector $\mathrm{u}$ of $(\lam, j)$ to be
$$\mathrm{u}=\begin{cases}
              \frac{\sqrt{2}}{2}u, & \mbox{\rm if } X=C_l^{(1)}; \\
              \sqrt{2}u, & \mbox{\rm if } X=D_{l+1}^{(2)}; \\
              u, & \mbox{\rm otherwise}.
            \end{cases}$$
\end{definition}

Now we are in a position to give the main result of this subsection.

\begin{lemma}\label{3.27}
Let $0\leq j\leq l$ and let $(\lam, j)$ be a core abacus with $w_{\lam, j}=t_{\mathrm{q}^\lam}\overline{w}$. Then 
$$\mathrm{u}^\lam=\frac{a_j^{\vee}}{a_0^{\vee}}\mathrm{q}^\lam+\overline{w}(\omega_{j}).$$
\end{lemma}

\begin{proof}
We prove by induction on the length of $w_{\lam, j}$ type by type. 
It is easy to check that the theorem holds if $\ell(w_{\lam, j})=0$. 
Take an arbitrary core abacus $(\lam, j)$ with $\ell(w_{\lam, j})=k>0$. 
It is clear that there exists $(\mu, j)$ such that $w_{\lam, j}=\sigma_iw_{\mu, j}$.

\smallskip

\noindent(1) {\em Type $C_l^{(1)}$}. We divide the proof into two cases according to the value of $i$.

{\em Case 1.} $i=0$. Note that $\sigma_0=t_{\mathrm{q}^{\sigma_0}}\overline{\sigma}_0$, where $\mathrm{q}^{\sigma_0}=(\sqrt{2}, 0, \dots, 0)$ 
and $\overline{\sigma}_0=\sigma_1 \cdots \sigma_{l-1}\sigma_l\sigma_{l-1}\cdots \sigma_1.$ 
Decomposing $w(\lam, j)$ as semidirect product $t_{\mathrm{q}^{\sigma_0}+\overline{\sigma}_0(\mathrm{q}^{\mu})}\overline{\sigma}_0\overline{w^{\mu}}$,
we have
\begin{align*}
\frac{a_j^{\vee}}{a_0^{\vee}}\mathrm{q}^\lam+\overline{w^\lam}(\omega_{j})
&=\frac{a_j^{\vee}}{a_0^{\vee}}(\mathrm{q}^{\sigma_0}+\overline{\sigma}_0(\mathrm{q}^{\mu}))+\overline{\sigma}_0\overline{w^{\mu}}(\omega_{j})\\ &=\frac{a_j^{\vee}}{a_0^{\vee}}\mathrm{q}^{\sigma_0}+\overline{\sigma}_0(\frac{a_j^{\vee}}{a_0^{\vee}}\mathrm{q}^{\mu}+\overline{w^{\mu}}(\omega_{j})) \\
& =\frac{a_j^{\vee}}{a_0^{\vee}}\mathrm{q}^{\sigma_0}+\overline{\sigma}_0(\mathrm{u}^\mu)\,\,\,\,   \text{(by induction hypothesis)}\\
&=(\frac{\sqrt{2}a_j^{\vee}}{a_0^{\vee}},0,...,0)+(-\frac{\sqrt{2}}{2}u_{1}^\mu,...,\frac{\sqrt{2}}{2}u_{l}^\mu) \,\,\,\,   \text{(by Lemma \ref{3.26})}\\
&=(\frac{\sqrt{2}a_j^{\vee}}{a_0^{\vee}}-\frac{\sqrt{2}}{2}u_{1}^\mu,...,\frac{\sqrt{2}}{2}u_{l}^\mu)\\
&=\frac{\sqrt{2}}{2}(\frac{2a_j^{\vee}}{a_0^{\vee}}-u_{1}^\mu,...,u_{l}^\mu)\\
&=\sigma_0\mathrm{u}^{\mu}\,\,\,\,   \text{(by Lemma \ref{3.26})}\\
&=\mathrm{u}^\lam.
\end{align*}

{\em Case 2.} $i\neq 0$. It follows from $w_{\lam, j}=t_{\sigma_i(\mathrm{q}^{\mu})}\sigma_i\overline{w^{\mu}}$ that

\begin{align*}
\frac{a_j^{\vee}}{a_0^{\vee}}\mathrm{q}^\lam+\overline{w^\lam}(\omega_{j})
&=\frac{a_j^{\vee}}{a_0^{\vee}}\sigma_i(\mathrm{q}^{\mu})+\sigma_i\overline{w^{\mu}}(\omega_{j})\\
&=\sigma_i(\frac{a_j^{\vee}}{a_0^{\vee}}\mathrm{q}^{\mu}+\overline{w^{\mu}}(\omega_{j}))\\
&=\sigma_i(\mathrm{u}^\mu)=\mathrm{u}^\lam.
\end{align*}

\noindent (2) {\em Type $D_{l}^{(1)}.$} The proof is similar to Case 2 of Type $C_l^{(1)}$ if $i\neq 0$. So we only consider the case $i=0$. 
Note that $\overline{\sigma}_0=\sigma_2 \cdots \sigma_{l-1}\sigma_1 \cdots \sigma_{l-2}\sigma_l\sigma_{l-1} \cdots \sigma_1\sigma_{l-1} \cdots \sigma_1$
and $\mathrm{q}^{\sigma_0}=(1, 1, 0, \dots, 0)$ if this type. 
Then by decomposing $w_{\lam, j}$ as semidirect product $t_{\mathrm{q}^{\sigma_0}+\overline{\sigma}_0(\mathrm{q}^{\mu})}\overline{\sigma}_0\overline{w^{\mu}}$,
we have
\begin{align*}
\frac{a_j^{\vee}}{a_0^{\vee}}\mathrm{q}^\lam+\overline{w^\lam}(\omega_{j})
&=\frac{a_j^{\vee}}{a_0^{\vee}}\mathrm{q}^{\sigma_0}+\overline{\sigma}_0(\frac{a_j^{\vee}}{a_0^{\vee}}\mathrm{q}^{\mu}+\overline{w^{\mu}}(\omega_{j})) \\
& =\frac{a_j^{\vee}}{a_0^{\vee}}\mathrm{q}^{\sigma_0}+\overline{\sigma}_0(\mathrm{u}^\mu)\,\,\,\,   \text{(by induction hypothesis)}\\
&=(\frac{a_j^{\vee}}{a_0^{\vee}},\frac{a_j^{\vee}}{a_0^{\vee}},0,...,0)+(-u_{2}^\mu,-u_{1}^\mu,...,u_{l}^\mu) \,\,\,\,   \text{(by Lemma \ref{3.26})}\\
&=(\frac{a_j^{\vee}}{a_0^{\vee}}-u_{2}^\mu,\frac{a_j^{\vee}}{a_0^{\vee}}-u_{1}^\mu,...,u_{l}^\mu)\\
&=\sigma_0\mathrm{u}^{\mu}=\mathrm{u}^\lam\,\,\,\,   \text{(by Lemma \ref{3.26})}.
\end{align*}
	
\noindent (3) {\em Type $B_l^{(1)}.$} The proof is similar to type $D_l^{(1)}$ and the details is omitted.  
We only want to point out here that $\mathrm{q}^{\sigma_0}=(1, 1, 0, \dots, 0)$ 
and $$\overline{\sigma}_0=\sigma_1\sigma_2 \cdots \sigma_{l-1}\sigma_l\sigma_{l-1}\cdots \sigma_2
\sigma_1 \cdots \sigma_{l-1}\sigma_l\sigma_{l-1}\cdots \sigma_1.$$ 
	
\noindent (4) The proof of Types $D_{l+1}^{(2)}$ and $A_{2l}^{(2)}$ is similar to that of type $C_l^{(1)}$
and the proof of Type $A_{2l-1}^{(2)}$ is similar to that of type $B_l^{(1)}$.

The proof if completed.
\end{proof}

\begin{example}\rm{
Recall the core abacus $(\lam, 1)$ in Example \ref{3.4}. Note that $w_{\lam, 1}=\sigma_1\sigma_2\sigma_1\sigma_0\sigma_1$. 
It is easy to check by direct computation  that $w_{\lam, 1}=t_{\mathrm{q}}\overline{w}$, 
where $\mathrm{q}=(-\sqrt{2}, 0)$ and $\overline{w}=\sigma_1$. Then
$$\frac{a_j^{\vee}}{a_0^{\vee}}\mathrm{q}+\overline{w}(\omega_{j})=2\mathrm{q}+\sigma_1(\omega_{1})=2(-\sqrt{2},0)+(0,\sqrt{2})=\sqrt{2}(-2,1),$$ 
which is equal to $\mathrm{u}(\lam, 1)$.
}
\end{example}

\bigskip


\section{Diophantine equations arising from core abaci}

The main task of this section is to generalize the height formula established in \cite{STW}
to arbitrary fundamental weight and then connect it to certain Diophantine equations.

\subsection{Height formula of core abaci}

In order to prove the height formula, we consider the height change of a core abacus being acted by $\sigma_i$, 
which is a key step of the proof by induction.
We first study the action of $\sigma_0$, which is divided into three cases according to Lemma \ref{3.26}.

\begin{lemma}\label{4.1}
Let $(\lam, j)$ be a core abacus of type $X$.
Then $${\rm ht}_0(\sigma_{0}(\lambda,j))=
\begin{cases}
\smallskip
{\rm ht}_0(\lambda,j)+(\frac{a_j^{\vee}}{a_0^{\vee}}-u_{1}-u_2), & \text{\rm if $X=B_l^{(1)}, D_l^{(1)}$ or $A_{2l-1}^{(2)}$};\\
\smallskip
{\rm ht}_0(\lambda,j)+(\frac{a_j^{\vee}}{a_0^{\vee}}-2u_{1}), & \text{\rm if $X=A_{2l}^{(2)}$ or $D_{l+1}^{(2)}$};\\
{\rm ht}_0(\lambda,j)+(\frac{a_j^{\vee}}{a_0^{\vee}}-u_{1}), & \text{\rm if $X= C_l^{(1)}$}.
\end{cases}
$$ 
\end{lemma}

\begin{proof}
It follows from Lemma \ref{3.26} that if
$X=B_l^{(1)}, D_l^{(1)}$ or $A_{2l-1}^{(2)}$, then we have $$u(\sigma_0(\lam, j))=
(\frac{a_j^{\vee}}{a_0^{\vee}}-u_2, \frac{a_j^{\vee}}{a_0^{\vee}}-u_1, \dots, u_l)$$ and consequently the 
difference between the number of beads in ${\rm\bf U}(\sigma_0(\lam, j))$ and that in ${\rm\bf U}(\lam, j)$
is $2(\frac{a_j^{\vee}}{a_0^{\vee}}-u_1-u_2)$. This implies that we need to do $\frac{a_j^{\vee}}{a_0^{\vee}}-u_1-u_2$
times of $f_0$ (or $e_0$) actions from $(\lam, j)$ to $\sigma_0(\lam, j)$, which is exactly as desired.

If $X=A_{2l}^{(2)}$ or $D_{l+1}^{(2)}$, then by Lemma \ref{3.26}, $u(\sigma_0(\lam, j))=
(\frac{a_j^{\vee}}{a_0^{\vee}}-u_1, \dots, u_l)$. Hence the difference between the number of beads 
 in ${\rm\bf U}(\sigma_0(\lam, j))$ and that in ${\rm\bf U}(\lam, j)$ is $\frac{a_j^{\vee}}{a_0^{\vee}}-2u_1$.
We divide the proof into three cases according to $j$. 

\begin{enumerate}
 \item [\rm Case 1.] Let $j=0$. Note that in this case $(\lam, 0)$ is a half abacus and 
 $\frac{a_j^{\vee}}{a_0^{\vee}}=\frac{a_0^{\vee}}{a_0^{\vee}}=1$. 
 Without loss of generality, we assume that $u_1\leq 0$. This implies that the leftmost position in $(\lam, 0)$ is empty.
 Then we need to do $1-2u_1$ times $f_0$,
 one of which is putting a bead in the leftmost position in $(\lam, 0)$.
 \item [\rm Case 2.] Let $1\leq j\leq l$. 
 Clearly, we need to do $\frac{a_j^{\vee}}{a_0^{\vee}}-2u_1$ times $f_0$ (or $e_0$) action from $(\lam, j)$ to $\sigma_0(\lam, j)$.
 \item [\rm Case 3.] Let $j=l$. In this case $(\lam, l)$ is a half abacus and 
 $\frac{a_l^{\vee}}{a_0^{\vee}}=1$. 
 Furthermore, the difference of beads $1-2u_1$ is even. Note that each $2f_0$ (or $2e_0$) action on $(\lam, l)$
 makes the number of beads in ${\rm\bf U}(\lam, l)$ increased (or decreased) by two and the proof of this case is complete.
\end{enumerate}

If $X = C_l^{(1)}$, then $u(\sigma_0(\lam, j))=
(\frac{2a_j^{\vee}}{a_0^{\vee}}-u_1, \dots, u_l)$, and the difference between the number of beads 
 in ${\rm\bf U}(\sigma_0(\lam, j))$ and that in ${\rm\bf U}(\lam, j)$ is $\frac{2a_j^{\vee}}{a_0^{\vee}}-2u_1$,
 which corresponds to $\frac{a_j^{\vee}}{a_0^{\vee}}-u_1$ times of $f_0$ (or $e_0$) action on $(\lam, j)$.
\end{proof}

Next we study the effect of $\sigma_i$ on the height of a core abacus for $1\leq i<l$.

\begin{lemma}\label{4.2}
Given a core abacus $(\lam, j)$, then for $1\leq i<l$, 
$${\rm ht}_i(\sigma_{i}(\lambda,j))={\rm ht}_i(\lambda,j)+(u_i-u_{i+1}).$$ 
\end{lemma}

\begin{proof}
Without loss of generality, we assume that $u_i>u_{i+1}$.
We have from Lemma \ref{3.26} that $u(\sigma_i(\lam, j))=
(u_1, \dots u_{i+1}, u_i, \dots, u_l)$ if $1\leq i<l$. This implies that we need to do $u_i-u_{i+1}$ times of $f_i$
from $(\lam, j)$ to $\sigma_i(\lam, j)$ and then the proof is complete.
\end{proof}

The last one we need to consider is the action of $\sigma_l$.

\begin{lemma}\label{4.3}
Given a core abacus $(\lam, j)$, we have
$${\rm ht}_l(\sigma_{l}(\lambda,j))=
\begin{cases}
{\rm ht}_l(\lambda,j)+u_l, & \text{\rm if $X = A_{2l}^{(2)}, A_{2l-1}^{(2)}$ or $C_l^{(1)}$};\\
{\rm ht}_l(\lambda,j)+2u_l, & \text{\rm if $X = B_l^{(1)}$ or $D_{l+1}^{(2)}$};\\
{\rm ht}_l(\lambda,j)+(u_{l-1}+u_{l}), & \text{\rm if $X = D_{l}^{(1)}$}.
\end{cases}
$$ 
\end{lemma}

\begin{proof}
If $X = D_{l}^{(1)}$, we have $u(\sigma_l(\lam, j))=
(u_1, \dots, -u_{l}, -u_{l-1})$ by Lemma \ref{3.26}. As a result, the difference between the number of beads in 
${\rm\bf U}(\sigma_l(\lam, j))$ and that in ${\rm\bf U}(\lam, j)$ is $-2(u_{l-1}+u_l)$, which derives from $u_{l-1}+u_l$
times of $f_l$ (or $e_l$) action on $(\lam, j)$.

We now assume that $X \neq D_{l}^{(1)}$. Then $u(\sigma_l(\lam, j))=
(u_1, \dots,  u_{l-1}, -u_{l})$ by Lemma \ref{3.26}. We divide the proof into two cases according to
whether $l+1$ is an $l$ index of $(\lam, j)$.
\begin{enumerate}
\item [\rm Case 1.] $X= A_{2l}^{(2)}, A_{2l-1}^{(2)}$ or $C_l^{(1)}$. Clearly, the difference of number of beads is $-2u_l$,
which derives from $u_l$ times of $f_l$ (or $e_l$) action on $(\lam, j)$.
\item [\rm Case 2.] $X =  B_l^{(1)}$ or $D_{l+1}^{(2)}$. Since $l+1$ is an $l$ index, 
it is denoted by $2f_l$ that moving a bead at position $(l, 2k)$ in $(\lam, j)$ to an empty position $(-l, 2k+1)$.
Hence the difference $-2u_l$ of number of beads derives from $2u_{l}$
times of $f_l$ (or $e_l$) action on $(\lam, j)$.
\end{enumerate}
\end{proof}

By using Lemmas \ref{4.1}-\ref{4.3}, we can give a height formula of a core abaci calculated by weighted Uglov vector. 

\begin{lemma}\label{4.4}
Given a core abacus $(\lam, j)$ with the weighted Uglov vector $\mathrm{u}$, then
$${\rm ht}_i(\lambda, j)=(\frac{a_0^{\vee}}{a_j^{\vee}}\frac{a_i}{2}\mathrm{u}-\omega_{i}^\vee, \mathrm{u})-
(\frac{a_0^{\vee}}{a_j^{\vee}}\frac{a_i}{2}\omega_{j}-\omega_{i}^\vee,\omega_{j})$$
and consequently
$${\rm ht}(\lambda, j)=\frac{a_0^{\vee}}{a_j^{\vee}}\frac{h}{2}(|\mathrm{u}|^2-|\omega_{j}|^2)-(\mathrm{u}-\omega_{j}, \rho^{\vee}),$$
where $\rho^{\vee}=\sum_{i=0}^{l}\omega_{i}^{\vee}$.
\end{lemma}

\begin{proof}
We shall only prove the ${\rm ht}_i$ formula by induction on the length of the affine Weyl group element $w_{\lam, j}$.
It is easy to check the lemma holds when $\ell(w_{\lam, j})=0$. Now assume that $\ell(w_{\lam, j})=k>0$, where $(\lam, j)$ is a core abacus.
Then there exists some core abacus $(\mu, j)$ and $0\leq r\leq l$ such that $w_{\lam, j}=\sigma_rw_{\mu, j}$. 
Write by $\mathrm{u}$ and $\mathrm{u}^{\sigma_r}$ the weighted Uglov vector of $(\mu, j)$ and $(\lam, j)$, respectively.
Denote by 
$$\Delta_i :=(\frac{a_0^{\vee}}{a_j^{\vee}}\frac{a_i}{2}\mathrm{u}^{\sigma_r}-\omega_{i}^{\vee}, \mathrm{u}^{\sigma_r})-
(\frac{a_0^{\vee}}{a_j^{\vee}}\frac{a_i}{2}\mathrm{u}-\omega_{i}^{\vee}, \mathrm{u}).$$ Then it is easy to check that 
$$\Delta_i =\frac{a_0^{\vee}}{a_j^{\vee}}\frac{a_i}{2}[(\mathrm{u}^{\sigma_r}, \mathrm{u}^{\sigma_r})-(\mathrm{u}, \mathrm{u})]+(\mathrm{u}-\mathrm{u}^{\sigma_r},\omega_{i}^{\vee}).$$
Clearly we only need to prove that $\Delta_i={\rm ht}_i(\lam,j)-{\rm ht}_i(\mu,j)$ for each $0\leq i\leq l$.

\smallskip

We will only give the proof of type $C_l^{(1)}$ because the proof of the other types are similar. We first consider $r=0$. By Lemma \ref{3.26}, $$\mathrm{u}^{\sigma_r}=\frac{\sqrt{2}}{2}(\frac{2a_j^{\vee}}{a_0^{\vee}}-u_1, \ldots, u_l),$$ 
and thus $$\Delta_i = (\frac{a_j^{\vee}}{a_0^{\vee}}-u_1)(a_i-\sqrt{2}\omega_{i1}^{\vee}),$$ where $\omega_{i1}^{\vee}$ is the first component of $\omega_{i}^{\vee}$. 
Recall that 
$$a_i=\begin{cases}
1, & \mbox{if } i=0, l \\
2, & \mbox{if } 1\leq i\leq l
\end{cases}$$
and 
$$\omega_{i1}^{\vee}=\begin{cases}
0, & \mbox{if } i=0 \\
\sqrt{2}, & \mbox{if } 1\leq i\leq l-1 \\
\frac{\sqrt{2}}{2}, & \mbox{if } i=l,
\end{cases}$$
and thus
$$\Delta_i=\begin{cases}
\frac{a_j^{\vee}}{a_0^{\vee}}-u_1, & \mbox{if } i=0 \\
0, & \mbox{otherwise},
\end{cases}$$
which is just the difference of ${\rm ht}_i$ between $(\lam, j)$ and $(\mu, j)$ by Lemma \ref{4.1}.

If $1\leq r\leq l-1$, then by Lemma \ref{3.26} $$\mathrm{u}^{\sigma_r}=\frac{\sqrt{2}}{2}(u_1,...,u_{r+1},u_{r},...,u_l),$$ and therefore
$$\Delta_i = \frac{\sqrt{2}}{2}(\omega_{ir}^{\vee}-\omega_{i,r+1}^{\vee})(u_r-u_{r+1}).$$ According to the definition of $\omega_{i}^{\vee}$, we have
$$\omega_{ir}^{\vee}-\omega_{i,r+1}^{\vee}=
\begin{cases}
  0-0, & \mbox{if } i<r \\
  \sqrt{2}-0, & \mbox{if } i=r \\
  \sqrt{2}-\sqrt{2}, & \mbox{if } i>r.
\end{cases}$$
This implies that
$$\Delta_i = 
\begin{cases}
u_r-u_{r+1}, & \mbox{if } i=r \\
0, & \mbox{otherwise},
\end{cases}$$
which is equal to ${\rm ht}_i(\lam,j)-{\rm ht}_i(\mu,j)$ for each $0\leq i\leq l$ by Lemma \ref{4.2}.

If $r=l$, then $u^{\sigma_l}=\frac{\sqrt{2}}{2}(u_1,...,-u_l)$ and consequently 
$$\Delta_i = \sqrt{2}u_l\omega_{il}^{\vee}=
\begin{cases}
u_l, & \mbox{if } i=l \\
0, & \mbox{otherwise}.
\end{cases}$$
By Lemma \ref{4.3}, we have $\Delta_i={\rm ht}_i(\lam,j)-{\rm ht}_i(\mu,j)$ for each $0\leq i\leq l$.
\end{proof}

The main significance of Lemma \ref{4.4} in this paper is the Diophantine equations arising from it. 
But before we study the equations in the next subsection, we want to point out that the height of a core $(\lam, j)$
is just the $\Lambda_j$-atomic length of $w_{\lam, j}$. This implies that we have established a combinatorics model for atomic length to a certain extent.

\begin{theorem}\label{4.5}
Let $(\lam, j)$ be a core abacus. Then ${\rm ht}(\lam,j)=\mathscr{L}_{\Lambda_j}(w_{\lam, j})$.
\end{theorem}

\begin{proof}
Assume that $w_{\lam, j}=t_{\mathrm q}\overline{w}$.
\begin{align*}
{\rm ht}(\lam,j)
&=\frac{a_0^{\vee}}{a_j^{\vee}}\frac{h}{2}(|\mathrm{u}|^2-|\omega_{j}|^2)-(\mathrm{u}-\omega_{j},\rho^{\vee})\quad\text{(by Lemma \ref{4.4})}\\ 
&=\frac{a_0^{\vee}}{a_j^{\vee}}\frac{h}{2}((\frac{a_j^{\vee}}{a_0^{\vee}}{\mathrm q}+\overline{w}(\omega_{j}),\frac{a_j^{\vee}}{a_0^{\vee}}{\mathrm q}+\overline{w}(\omega_{j}))-|\omega_{j}|^2)-\\
& \quad\,\,\textup{ht}(\frac{a_j^{\vee}}{a_0^{\vee}}{\mathrm q}+\overline{w}(\omega_{j})-\omega_{j})\quad\text{(by Lemma \ref{3.27})}\\
&=\frac{a_j^{\vee}}{a_0^{\vee}}\frac{h}{2}|{\mathrm q}|^2+h({\mathrm q},\overline{w}(\omega_{j}))-
\frac{a_j^{\vee}}{a_0^{\vee}}\textup{ht}({\mathrm q})+\textup{ht}(\omega_{j}-\overline{w}(\omega_{j}))\\
&=\frac{a_j^{\vee}}{a_0^{\vee}}(\frac{h}{2}|{\mathrm q}|^2-\textup{ht}({\mathrm q}))+h({\mathrm q},\overline{w}(\omega_{j}))+\textup{ht}(\omega_{j}-\overline{w}(\omega_{j}))\\
&=\mathscr{L}_{\Lambda_j}(w_{\lam, j})\quad\text{(by \cite[Proposition 3.2]{BCG})}.
\end{align*}
The proof is complete.
\end{proof}

\begin{remarks}\rm{
(1) In \cite{BCG}, Brunat, Chapelier-Laget and Gerber raised an open problem to 
find a combinatorial model in which the size of partitions is given by the atomic length.
The above corollary can be viewed as an answer.

(2) Many experts in Combinatorics studied the Nekrasov–Okounkov formula and related topics. We refer the reader to \cite{LW, NO, P, W} for some examples. 
Among them, Lecouvey and Wahiche \cite{LW} expanded the affine Weyl denominator formulae as signed $e^{-\delta}$-series of ordinary Weyl characters running over
the affine Grassmannian and proved that the grading in $e^{-\delta}$ coincides with the (dual) atomic length. 
We think Theorem \ref{4.5} and Lemma \ref{3.12} should contribute to making progress on the relevant fields.
}
\end{remarks}

\begin{example}\label{4.7}\rm{
Let us consider the core $(\lam, 1)=((4,2,1,1,1,1,1), 1)$, which appeared in Example \ref{3.4}. Recall that $\mathrm{u}(\lam, 1)=\sqrt{2}(-2, 1)$.
Then by Lemma \ref{4.4}, we have
$${\rm ht}_0(\lambda)=(\frac{\sqrt{2}}{4}(-2,1),\sqrt{2}(-2,1))-(\frac{\sqrt{2}}{4}(1,0),\sqrt{2}(1,0))=2;$$
$${\rm ht}_1(\lambda)=(\frac{\sqrt{2}}{4}(-2,1)-\frac{\sqrt{2}}{2}(1,0),\sqrt{2}(-2,1))-(\frac{\sqrt{2}}{4}(1,0)-\frac{\sqrt{2}}{2}(1,0),\sqrt{2}(1,0))=5;$$
$${\rm ht}_2(\lambda)=(\frac{\sqrt{2}}{4}(-2,1)-\frac{\sqrt{2}}{2}(1,1),\sqrt{2}(-2,1))-(\frac{\sqrt{2}}{4}(1,0)-\frac{\sqrt{2}}{2}(1,1),\sqrt{2}(1,0))=4.$$

Filling the residue in each node of the Young diagram, we have

\center{\young(1221,01,0,1,2,2,1)}

which coincides with the result obtained by computing. Moreover,
recall that $w_{\lam, 1}=s_1s_2s_1s_0s_1$. It is easy to check that the atomic length of $w_{\lam, 1}$ is just the height 11.
}
\end{example}

Let us conclude this section by giving an application of Lemma \ref{4.4} on conjugation of a core abacus.

\begin{corollary}\label{4.8}
Let $(\lam, j)$ and $(\mu, s)$ be core abaci of the same type $C_l^{(1)}$, $D_l^{(1)}$ or $D_{l+1}^{(2)}$ with $j$ satisfying $(3.2.1)$. 
Suppose that the Uglov vector of $(\lam, j)$ is $(u_1, \dots, u_l)$. 
Then the Uglov vector of $(\mu, s)$ is $(1-u_l, \dots, 1-u_1)$ if and only if $(\mu, s)=(\lam', l-j)$.
\end{corollary}

\begin{proof}
Since the proof of all types is similar, we only consider type $C_l^{(1)}$.

``$\Longrightarrow$"  Assume that $(\lam, j)$ and $(\mu, s)$ are core abaci of type $C_l^{(1)}$.
By Lemma \ref{3.5}, there are $j$ odd components in $(u_1, \dots, u_l)$, 
and thus the number of odd components in $(1-u_l, \dots, 1-u_1)$ is $l-j$, that is, $s=l-j$.
Then by Lemma \ref{3.8} it is sufficient to prove that the number of $f_i$-actions on $(\lam, j)$
is equal to that of $f_{l-i}$-actions on $(\mu, l-j)$ for all $i=0,1, \dots, l$. 
Note that this can be computed by Lemma \ref{4.4}. We divide the computation into 4 cases according to the value of $i$.
\begin{enumerate}
\item [\rm Case 1.] $i=0$. ${\rm ht}_0(\lam, j)=\frac{1}{4}(u_1^2+\dots+u_l^2-j)={\rm ht}_l(\mu, l-j).$

\smallskip

\item [\rm Case 2.] $1\leq i\leq l-1$ and $j\leq i$. $${\rm ht}_i(\lam, j)=\frac{1}{2}(u_1^2+\dots+u_l^2)-(u_1+\dots+u_i)+\frac{j}{2}={\rm ht}_{l-i}(\mu, l-j).$$
\item [\rm Case 3.] $1\leq i\leq l-1$ and $j>i$.
		$${\rm ht}_i(\lam, j)=\frac{1}{2}(u_1^2+\dots+u_l^2)-(u_1+\dots+u_i)+\frac{2i-j}{2}={\rm ht}_{l-i}(\mu, l-j).$$
\item [\rm Case 4.] $i=l$.
		${\rm ht}_l(\lam, j)=\frac{1}{4}(u_1^2+\dots+u_l^2+j)-\frac{1}{2}(u_1+\dots+u_l)={\rm ht}_0(\mu, l-j).$
\end{enumerate}

\smallskip

``$\Longleftarrow$" Suppose that the Uglov vector of $(\lam', l-j)$ is $(v_1, \dots, v_l)$.
Then we have from Lemma \ref{3.8} and Lemma \ref{4.4} that

\smallskip

\noindent$\begin{cases}
\frac{1}{4}(u_1^2+\dots+u_l^2-j)&=	\frac{1}{4}(v_1^2+\dots+v_l^2+l-j)-	\frac{1}{2}(v_1+\dots+v_l)\\
\frac{1}{2}(u_1^2+\dots+u_l^2-j)-u_1+1&=	\frac{1}{2}(v_1^2+\dots+v_l^2+l-j)-	(v_1+\dots+v_{l-1})\\
\quad\quad\dots\quad\quad\dots\quad\quad\dots\quad\quad\dots & \quad\quad\dots\quad\quad\dots\quad\quad\dots\quad\quad\dots\quad\quad\dots\\
\frac{1}{2}(u_1^2+\dots+u_l^2+j)-(u_1+\dots+u_{l-1})&=	\frac{1}{2}(v_1^2+\dots+v_l^2)-	v_1-\frac{1}{2}(l-j)+1\\
\frac{1}{4}(u_1^2+\dots+u_l^2+j)-\frac{1}{2}(u_1+\dots+u_l)&=\frac{1}{4}(v_1^2+\dots+v_l^2)-\frac{1}{4}(l-j)\\
\end{cases}$

\smallskip

\noindent Consequently, $(v_1, \dots, v_l)=(1-u_l, \dots, 1-u_1)$.
\end{proof}

\medskip


\subsection{Diophantine equations}
Let us consider Diophantine equations arising naturally from the height formula of core abaci by direct computation. 
These equations will be called Diophantine equations of affine types.

We list equalities associated with core abaci type by type as follows without specific process.
Then Diophantine equations can be read out from the equalities easily. 
Note that all the equations are of the form
$\sum_ix_i^2=\mathtt{a}N+\mathtt{b}$, where $N$ is the height of a $j$-core (half) abacus of type $X$ and $\mathtt{a}, \mathtt{b}$ are determined by $l, j$. 

\begin{lemma}\label{4.9}
Let $(\lam, j)$ be a core abacus of type $A_{2l-1}^{(2)}$ with ${\rm ht}(\lambda, j)=N$. 
If $j=0, 1$, then
$$8(2l-1)N+\frac{l(2l+1)(2l-1)}{3}=\sum_{i=1}^l \left( 2(2l-1)u_i -( 2(l-i)+1) \right)^2;$$
if $2\leq j\leq l$, then
$$4(2l-1)N+\frac{l(2l+1)(2l-1)}{3}-j(2l-1)(2l-2j+1)=\sum_{i=1}^l \left( (2l-1)u_i -( 2(l-i)+1) \right)^2.$$ 
\end{lemma}

\begin{lemma}\label{4.10}
Let $(\lam, j)$ be a core abacus of type $A_{2l}^{(2)}$  with ${\rm ht}(\lambda, j)=N$. 
If $j=0$, then	
$$8(2l+1)N+\frac{l(2l+1)(2l-1)}{3}=\sum_{i=1}^l \left( 2(2l+1)u_i -( 2(l-i)+1) \right)^2;$$
if $1\leq j\leq l$, then
$$4(2l+1)N+\frac{l(2l+1)(2l-1)}{3}-j(2l+1)(2l-2j-1)=\sum_{i=1}^l \left( (2l+1)u_i -( 2(l-i)+1) \right)^2.$$
\end{lemma}

\begin{lemma}\label{4.11}
Let $(\lam, j)$ be a core abacus of type $B_l^{(1)}$ with ${\rm ht}(\lambda, j)=N$. Then
$$4lN+\frac{l(l+1)(2l+1)}{6}=\sum_{i=1}^l \left( 2lu_i - (l-i+1) \right)^2, \,\,{\text if}\,\, j=0,1;$$
$$2lN+\frac{l(l+1)(2l+1)}{6}-jl(l-j+1)=\sum_{i=1}^l \left( lu_i - (l-i+1) \right)^2, \,\,{\text if}\,\, 2\leq j\leq l-1;$$
$$	4lN+\frac{l(l+1)(2l+1)}{6}-l^2=\sum_{i=1}^l \left( 2lu_i - (l-i+1) \right)^2, \,\,{\text if} \,\,j=l.$$
\end{lemma}

\begin{lemma}\label{4.12}
Let $(\lam, j)$ be a core abacus of type $C_l^{(1)}$  with ${\rm ht}(\lambda, j)=N$. 
Then for $0\leq j\leq l$
$$8lN+\frac{l(2l+1)(2l-1)}{3}-4lj(l-j)=\sum_{i=1}^l \left( 2lu_i -( 2(l-i)+1) \right)^2.$$
\end{lemma}

\begin{lemma}\label{4.13}
Let $(\lam, j)$ be a core abacus of type $D_l^{(1)}$ with ${\rm ht}(\lambda, j)=N$.	
If $j=0, 1, l-1, l$, then
$$4(l-1)N+\frac{(l-1)l(2l-1)}{6}=\sum_{i=1}^l \left( 2(l-1)u_i -(l-i) \right)^2;$$
if $2\leq j\leq l-2$, then
$$2(l-1)N+\frac{(l-1)l(2l-1)}{6}-j(l-1)(l-j)=\sum_{i=1}^l \left((l-1)u_i -(l-i) \right)^2.$$
\end{lemma}

\begin{lemma}\label{4.14}
Let $(\lam, j)$ be a core abacus of type $D_{l+1}^{(2)}$  with ${\rm ht}(\lambda, j)=N$.
If $j=0, l$, then
$$4(l+1)N+\frac{l(l+1)(2l+1)}{6}=\sum_{i=1}^l \left( 2(l+1)u_i -(l-i+1) \right)^2;$$
if $1\leq j\leq l-1$, then
$$2(l+1)N+\frac{l(l+1)(2l+1)}{6}-j(l+1)(l-j)=\sum_{i=1}^l \left( (l+1)u_i -(l-i+1) \right)^2.$$
\end{lemma}

According to the above lemmas, we can define affine transformation $F$ on Uglov vectors as follows.

\begin{definition}\label{4.15}
Let $(\lam, j)$ be a core abacus of type $X$ with Uglov vector $u$. Define 
$$F(u)=\mathtt{k} u-\mathrm{c},$$ where $\mathtt{k}$ and $\mathrm{c}=(\mathrm{c}_1, \dots, \mathrm{c}_l)$ are defined as follows:
$$\mathtt{k}=\begin{cases}
\smallskip
l, &\text{\rm if $X=B_l^{(1)}, 2\leq j\leq l-1$};\\
\smallskip
l-1, &\text{{\rm if} $X=D_l^{(1)}$, $2\leq j\leq l-2$};\\
\smallskip
l+1, &\text{{\rm if} $X=D_{l+1}^{(2)}$, $1\leq j\leq l-1$};\\
\smallskip
2l, &\text{{\rm if} $X=B_l^{(1)}, j=0, 1, l$ {\rm or type} $C_l^{(1)}$};\\
\smallskip
2l-1, &\text{{\rm if} $X=A_{2l-1}^{(2)}, 2\leq j\leq l$};\\
\smallskip
2l-2, &\text{{\rm if} $X=D_l^{(1)}, j=0,1, l-1, l$};\\
\smallskip
2l+1, &\text{{\rm if} $X=A_{2l}^{(2)}, j\neq 0$};\\
\smallskip
2l+2, &\text{{\rm if} $X=D_{l+1}^{(2)}, j=0, l$};\\
\smallskip
4l-2, &\text{{\rm if} $X=A_{2l-1}^{(2)}, j=0, 1$};\\
\smallskip
4l+2, &\text{{\rm if} $X=A_{2l}^{(2)}, j=0$},\\
\end{cases}$$
$$\mathrm{c}_i=\begin{cases}
\smallskip
2(l-i)+1, &\text{{\rm if} $X=A_{2l-1}^{(2)}, A_{2l}^{(2)} \,{\rm or}\, C_l^{(1)}$};\\
\smallskip
l-i+1, &\text{{\rm if} $X=B_l^{(1)}$, $D_{l+1}^{(2)}$};\\
\smallskip
l-i, &\text{{\rm if} $X=D_l^{(1)}$}.\\
\end{cases}$$
\end{definition}

Combining Lemmas \ref{4.9}-\ref{4.14} yields the following theorem.

\begin{theorem}\label{4.16}
Let $(\lam, j)$ be a type $X$ core abacus with ${\rm ht}(\lambda, j)=N$. 
Then $F(u)$ is a solution of the corresponding affine Diophantine equation $$x_1^2+x_2^2+\dots+x_l^2=\mathtt{a}N+\mathtt{b},$$
where $\mathtt{a}$ and $\mathtt{b}$ are defined as follows.
$$\mathtt{a}=\begin{cases}
\smallskip
2l, & \text{{\rm if} $X=B_l^{(1)}, 2\leq j\leq l-1$};\\
\smallskip
2(l-1), & \text{{\rm if} $X=D_l^{(1)}$, $2\leq j\leq l-2$};\\
\smallskip
2(l+1), & \text{{\rm if} $X=D_{l+1}^{(2)}$, $2\leq j\leq l-1$};\\
\smallskip
4l, & \text{{\rm if} $X=B_l^{(1)}, j=0, 1, l$};\\
\smallskip
4(l-1), & \text{{\rm if} $X=D_l^{(1)}, j=0,1, l-1, l$};\\
\smallskip
4(l+1), & \text{{\rm if} $X=D_{l+1}^{(2)}, j =0, l$};\\
\smallskip
8l, & \text{{\rm if} $X=C_{l}^{(1)}$};\\
\smallskip
8l-4, & \text{{\rm if} $X=A_{2l-1}^{(2)}, 2\leq j\leq l$};\\
\smallskip
8l+4, & \text{{\rm if} $X=A_{2l}^{(2)}, j\neq 0$};\\
\smallskip
16l-8, & \text{{\rm if} $X=A_{2l-1}^{(2)}, j=0, 1$};\\
\smallskip
16l+8, & \text{{\rm if} $X=A_{2l}^{(2)}, j = 0$,}
\end{cases}$$
$$\mathtt{b}=\begin{cases}
\smallskip
(2l+1)(\frac{1}{3}l(2l-1)-j(2l-2j-1)), & \text{{\rm if} $X=A_{2l}^{(2)}$};\\
\smallskip
(2l-1)(\frac{1}{3}l(2l+1)-j(2l-2j+1)), & \text{{\rm if} $X=A_{2l-1}^{(2)}, j\neq 1$};\\
\smallskip
\frac{1}{3}l(2l+1)(2l-1), & \text{{\rm if} $X=A_{2l-1}^{(2)}, j= 1$};\\
\smallskip
\frac{1}{3}l(2l+1)(2l-1)-4lj(l-j), & \text{{\rm if} $X=C_l^{(1)}$};\\
\smallskip
\frac{1}{6}l(l+1)(2l+1)-jl(l-j+1), & \text{{\rm if} $X=B_l^{(1)}$, $j\neq 1$};\\
\smallskip
\frac{1}{6}l(l+1)(2l+1), & \text{{\rm if} $X=B_l^{(1)}$, $j= 1$};\\
\smallskip
(l-1)(\frac{1}{6}l(2l-1)-j(l-j)), & \text{{\rm if} $X=D_l^{(1)}, j\neq 1, l-1$};\\
\smallskip
\frac{1}{6}(l-1)l(2l-1), & \text{{\rm if} $X=D_l^{(1)}, j= 1, l-1$};\\
\smallskip
(l+1)(\frac{1}{6}l(2l+1)-j(l-j)), & \text{{\rm if} $X=D_{l+1}^{(2)}$}.\\
\end{cases}$$
\end{theorem}

\begin{remark}\rm{
In general, core abaci only provide partially parametrization for the solutions of the corresponding Diophantine equation.
We will consider in the next section under what conditions the parametrization is complete.}
\end{remark}

\bigskip


\section{Solutions of Diophantine equations and core abaci}

The main task of this section is to answer the following question: 
given a solution orbit $\rm O$ of a classical affine Lie type Diophantine equation, 
under what conditions $\rm O$ is parameterized by core abaci? After giving some general results, 
we obtain some special classes Diophantine equations completely parameterized by core abaci.
Consequently, some closed formulae on the number of certain core abaci are given.
We shall also raise two open problems. 

\subsection{Some general results}
It is well known that the Weyl group $W$ of type $B_l$ acts on the solution set of affine type 
Diophantine equation $x_1^2+x_2^2+\dots+x_l^2=\mathtt{a}N+\mathtt{b}$ naturally 
(see Theorem 4.16 for notations $\mathtt{a}$ and $\mathtt{b}$). For a given solution $t$, 
define the associated solution orbit ${\rm O}_t=\{w.t\mid w\in W\}$.

For the convenience of description, we give a definition as below.
\begin{definition}
Let $t$ be a solution of an affine type Diophantine equation $x_1^2+x_2^2+\dots+x_l^2=\mathtt{a}N+\mathtt{b}.$ We say $t$ is parameterized by a core abacus if
there exists certain height $N$ core abacus $(\lam, j)$ with the Uglov vector $(u_1, u_2, \dots, u_l)$ such that $F(u)=t$.
A solution orbit ${\rm O}_t$ is called to be parameterized by core abaci if there is a solution in ${\rm O}_t$ being parameterized by a core abacus.
Denote by ${\rm O}^{\rm p}_t$ the solutions in ${\rm O}_t$ being parameterized by core abaci.
The Diophantine equation is said to be completely parameterized by core if each solution orbit being parameterized by core abaci.
\end{definition}

Let us give an easy criterion for a solution being parameterized by a core abacus. The proof is omitted here. 
Recall notations $\mathtt{k}$ and $\mathrm{c}_i$ in Definition \ref{4.15}.

\begin{lemma}\label{6.1}
A solution $t$ of an affine type Diophantine equation $x_1^2+x_2^2+\dots+x_l^2=\mathtt{a}N+\mathtt{b}$
is parameterized by a core abacus $(\lam, j)$ if and only if for $i=1, \dots, l$
\begin{enumerate}
\item [{\rm (1)}] $\frac{t_i+\mathrm{c}_i}{\mathtt{k}}\in\mathbb{Z}$
and the number of odd $\frac{t_i+\mathrm{c}_i}{\mathtt{k}}$ is $j$, for type
$C_l^{(1)}$ with charge $j$, 
or type $A_{2l-1}^{(2)}$, $A_{2l}^{(2)}$, $B_l^{(1)}$, $D_l^{(1)}$ or $D_{l+1}^{(2)}$ with charge $j$ satisfying $\frac{a_j^\vee}{a_0^\vee}=2$;
\item [{\rm (2)}] $\frac{t_i+\mathrm{c}_i}{\mathtt{k}}\in\mathbb{Z}$, for type
  $A_{2l-1}^{(2)}$, $B_l^{(1)}$ or $D_l^{(1)}$ with charge $j$ being $0$ or $1$, or type $A_{2l}^{(2)}$ or $D_{l+1}^{(2)}$ with charge $j=0$;
\item [{\rm (3)}] $\frac{t_i+\mathrm{c}_i}{\mathtt{k}}\in\mathbb{Z}+\frac{1}{2}$, for type
$D_l^{(1)}$ with charge $l-1$ or $l$, or type $B_{l}^{(1)}$ or $D_{l+1}^{(2)}$ with charge $l$.
\end{enumerate}
\end{lemma}

Based on Lemma \ref{6.1}, we can give an answer of the question raised at the beginning of this section.

\begin{lemma}\label{6.2}
Let ${\rm O}_t$ be a solution orbit of an affine type $X$ Diophantine equation $$x_1^2+x_2^2+\dots+x_l^2=\mathtt{a}N+\mathtt{b}$$ with charge $j$, where 
$t=(t_1, \dots, t_l)\equiv (r_1, \dots, r_l) \pmod {2\mathtt{k}}$ with $r_i=0, \dots, \mathtt{k}$. 
Then ${\rm O}_t$ is parameterized by core abaci if for $i=1, \dots, l$,
\begin{enumerate}
\item [{\rm (1)}] $r_i+\mathrm{c}_i=\mathtt{k}$ or $r_i=\mathrm{c}_i$ and the number of 
$r_i$ satisfying $r_i+\mathrm{c}_i=\mathtt{k}$ is exact $j$, for type
$C_l^{(1)}$ with charge $j$, or type $A_{2l-1}^{(2)}$, $A_{2l}^{(2)}$, $B_l^{(1)}$, $D_l^{(1)}$ 
or $D_{l+1}^{(2)}$ with charge $j$ satisfying $\frac{a_j^\vee}{a_0^\vee}=2$; 
\item [{\rm (2)}] $r_i=\mathrm{c}_i$, for type
  $A_{2l-1}^{(2)}$, $B_l^{(1)}$ or $D_l^{(1)}$ with charge $j$ being $0$ or $1$, or type $A_{2l}^{(2)}$ or $D_{l+1}^{(2)}$ with charge $j=0$;
\item [{\rm (3)}] $r_i+\mathrm{c}_i=\frac{\mathtt{k}}{2}$, for type
$D_l^{(1)}$ with charge $l-1$ or $l$, or type $B_{l}^{(1)}$ or $D_{l+1}^{(2)}$ with charge $l$.
\end{enumerate}
\end{lemma}

\begin{proof}
Since the all the proof of these three cases is similar, we shall only prove the first one.
Assume that $t_i=2\mathtt{k}k_i+r_i$ for all $i=1, \dots, l$.
Note that  $r_i=\mathtt{k}-\mathrm{c}_i$ or  $r_i=\mathrm{c}_i$.
Take $v=(v_1, \dots, v_l)\in {\rm O}_t$, where
$$v_i=\begin{cases}
       t_i, & \mbox{if } r_i=\mathtt{k}-\mathrm{c}_i; \\
       -t_i, & \mbox{otherwise}.
     \end{cases}$$
Then $$\frac{v_i+\mathrm{c}_i}{\mathtt{k}}=\begin{cases}
       2k_i+1, & \mbox{if } r_i=\mathtt{k}-\mathrm{c}_i; \\
       2k_i, & \mbox{otherwise}.
     \end{cases}$$
Note that the number of odd components in $(\frac{v_1+\mathrm{c}_1}{\mathtt{k}}, \dots, \frac{v_l+\mathrm{c}_l}{\mathtt{k}})$ is $j$. 
Then the result follows by Lemma \ref{6.1}.
\end{proof}

Next we study the relationship between two solution orbits. To this aim, we need to define an equivalence relation in $\mathbb{Z}^l$.

\begin{definition}\label{6.3}
Let $k\in \mathbb{Z}$ and $(b_1, \dots, b_l)\equiv (r_1, \dots, r_l)$, $(n_1, \dots, n_l)\equiv (r_1', \dots, r_l') \pmod k$, 
where $b_i, n_i\in \mathbb{Z}$ for $i=1, \dots, l$. Define $(b_1, \dots, b_l) \sim_k (n_1, \dots, n_l)$ if there exists 
$\sigma\in \mathfrak{S}_l$ such that $r_i=\sigma(r_i')$ or $r_i+\sigma(r_i')=k$ for each $i=1, \dots, l$.
\end{definition}

It is easy to check that $\equiv_k$ is an equivalence relation. 
The equivalence class including $(b_1, \dots, b_l)$ is denoted by $[(b_1, \dots, b_l)]$.
The following lemma can be obtained by direct calculations and we omit the proof here.

\begin{lemma}\label{5.4'}
Let $t, t'\in \rm O$ be two solutions of the equation $x_1^2+x_2^2+\dots+x_l^2=\mathtt{a}N+\mathtt{b}.$ 
Then $t \sim_{2\mathtt{k}} t'$.
\end{lemma}

Based on Lemma \ref{5.4'}, we can say that a solution orbit in an equivalence class in the sense of Definition \ref{6.3}.
It is easy to check that for arbitrary $(b_1, \dots, b_l)\in \mathbb{Z}^l$, there exists $(r_1, \dots, r_l)\sim_{2\mathtt{k}}(b_1, \dots, b_l)$ 
such that $0\leq r_i\leq \mathtt{k}$ for each $i=1, \dots, l$.
Then we have the following lemma about the relation between two solution orbits.

\begin{lemma}\label{6.4}
Let ${\rm O}_v$ and ${\rm O}_w$ be two solution orbits of an affine type $X$ Diophantine equation $$x_1^2+x_2^2+\dots+x_l^2=\mathtt{a}N+\mathtt{b}$$
with charge $j$, where $$(v_1, \dots, v_l)\equiv (r_1, \dots, r_l) \pmod {2\mathtt{k}},
$$ $$(w_1, \dots, w_l)\equiv (r_1', \dots, r_l') \pmod {2\mathtt{k}}$$ and
$$(r_1, \dots, r_l)\sim_{2\mathtt{k}}(r_1', \dots, r_l').$$
If $(r_1, \dots, r_l)$ satisfies one of the condition of Lemma \ref{6.2}, 
then ${\rm O}_w$ is parameterized by core abaci.
\end{lemma}

\begin{proof}
We assume that $(r_1, \dots, r_l)$ satisfies condition $(1)$ of Lemma \ref{6.2} because all the proof is similar.
Since $(r_1, \dots, r_l)\sim_{2\mathtt{k}}(r_1', \dots, r_l')$, by Definition \ref{6.3}, there exists $\sigma\in \mathfrak{S}_l$ such that 
$r_i=\sigma(r_i')$ or $r_i+\sigma(r_i')=2\mathtt{k}$ for each $i=1, \dots, l$. 
Then clearly 
$$\sigma(r_i')=\begin{cases}
	\mathtt{k}+\mathrm{c}_i, & \mbox{if }  \sigma(r_i')= 2\mathtt{k}-r_i\,\, \text{and} \,\,r_i+\mathrm{c}_i=\mathtt{k}; \\
	2\mathtt{k}-\mathrm{c}_i, & \mbox{if } \sigma(r_i')= 2\mathtt{k}-r_i\,\, \text{and} \,\,r_i+\mathrm{c}_i \neq\mathtt{k};\\
	\mathtt{k}-\mathrm{c}_i, & \mbox{if } \sigma(r_i')\neq 2\mathtt{k}-r_i \,\, \text{and}\,\,r_i+\mathrm{c}_i=\mathtt{k}; \\
     \mathrm{c}_i, & \mbox{if } \sigma(r_i') \neq 2\mathtt{k}-r_i\,\, \text{and} \,\,r_i+\mathrm{c}_i \neq \mathtt{k}.
\end{cases}$$
Take $z=(z_1, \dots, z_l)\in {\rm O}_w$, where
$$z_i=\begin{cases}
	-w_i, & \mbox{if }  \sigma(r_i')= 2\mathtt{k}-r_i\,\, \text{and} \,\,r_i+\mathrm{c}_i=\mathtt{k}; \\
	w_i, & \mbox{if } \sigma(r_i')= 2\mathtt{k}-r_i\,\, \text{and} \,\,r_i+\mathrm{c}_i \neq\mathtt{k};\\
	w_i, & \mbox{if } \sigma(r_i')\neq 2\mathtt{k}-r_i \,\, \text{and}\,\,r_i+\mathrm{c}_i=\mathtt{k}; \\
    -w_i, & \mbox{if } \sigma(r_i') \neq 2\mathtt{k}-r_i\,\, \text{and} \,\,r_i+\mathrm{c}_i \neq \mathtt{k}.
\end{cases}$$
Suppose that $w_i=2\mathtt{k}k_i+\sigma(r_i')$ for all $i=1, \dots, l$.
Then
$$\frac{z_i+\mathrm{c}_i}{\mathtt{k}}=\begin{cases}
	2k_i-1, & \mbox{if }  \sigma(r_i')= 2\mathtt{k}-r_i\,\, \text{and} \,\,r_i+\mathrm{c}_i=\mathtt{k}; \\
	2k_i+2, & \mbox{if } \sigma(r_i')=2\mathtt{k}-r_i\,\, \text{and} \,\,r_i+\mathrm{c}_i \neq\mathtt{k};\\
	2k_i+1, & \mbox{if } \sigma(r_i')\neq 2\mathtt{k}-r_i \,\, \text{and}\,\,r_i+\mathrm{c}_i=\mathtt{k}; \\
	2k_i, & \mbox{if } \sigma(r_i') \neq 2\mathtt{k}-r_i\,\, \text{and} \,\,r_i+\mathrm{c}_i \neq \mathtt{k}.
\end{cases}$$ 
By Lemma \ref{6.1}, the result follows.
\end{proof}

By Lemma \ref{6.4}, direct computation yields the following result, which can be used to simplify the related calculation in the next subsections.

\begin{corollary}\label{6.5}
Let ${\rm O}_v$ and ${\rm O}_w$ be two solution orbits of an affine type $X$ 
Diophantine equation $x_1^2+x_2^2+\dots+x_l^2=\mathtt{a}N+\mathtt{b}$ with $[v]=[w]$.
Then the number of solutions being parameterized by core abaci in ${\rm O}_v$ is equal to that in ${\rm O}_w$.
\end{corollary}

\medskip


\subsection{Special Diophantine equations}

Combining Lemmas \ref{6.2} and \ref{6.4}, we can study some special classes of Diophantine equations.

\begin{theorem}\label{6.6}
The following affine type Diophantine equations are completely parameterized by core abaci.
\begin{enumerate}
  \item [{\rm(1)}] $x^2+y^2=16N+10$; 
  \item [{\rm(2)}] $x^2+y^2=16N+2$;
  \item [{\rm(3)}] $x^2+y^2+z^2=24N+11$; 
  \item [{\rm(4)}] $x^2+y^2+z^2=6N+2$;
  \item [{\rm(5)}] $x_1^2+x_2^2+x_3^2+x_4^2=8N+6$;
  \item [{\rm(6)}] $x^2+y^2=12N+5$; 
  \item [{\rm(7)}] $x^2+y^2=6N+2$; 
  \item [{\rm(8)}] $x^2+y^2+z^2=8N+6$; 
  \item [{\rm(9)}] $x_1^2+x_2^2+x_3^2+x_4^2=6N+2$.
\end{enumerate}
\end{theorem}

\begin{proof}
We only prove (9) because the proof of all the others is similar. 

Note that equation (9) is of affine type $D_4^{(1)}$ with charge 2. Note that $\mathtt{k}=3$, $\mathrm{c}=(3, 2, 1, 0)$ in this case.
For a given nonnegative integer $N$, let ${\rm O}_t$ be a solution orbit of equation (9) with $t$ being a nonnegative solution and 
$(t_1, t_2, t_3, t_4)\equiv (r_1, r_2, r_3, r_4) \pmod {6}$ with $0\leq r_1\leq r_2\leq r_3 \leq r_4\leq 5$. 
Since $6N+2$ is even, an easy computation gives that all $t_i$ are all odd, all even or only two ones even.
Moreover, the residue $(r_1, r_2, r_3, r_4)$ has to be one of the
following 16 cases by a routine calculation:
$$\begin{array}{cccccc}
  (0, 0, 1, 1) & (0, 0, 1, 5) & (0, 0, 2, 2) & (0, 0, 2, 4) & (0, 0, 4, 4) & (0, 0, 5, 5)\\ 
  (0, 1, 2, 3) & (0, 1, 3, 4) & (0, 2, 3, 5) & (0, 3, 4, 5) & (1, 1, 3, 3) & (1, 3, 3, 5)\\
  (2, 2, 3, 3) & (2, 3, 3, 4) & (3, 3, 4, 4) & (3, 3, 5, 5)
\end{array}$$
If $r=(r_1,r_2,r_3,r_4)=(0,0,1,1)$, denoting $(0, 1, 1, 0)$ by $\overline{r}$, then $\overline{t}=(t_1, t_3, t_4, t_2)\in {\rm O}_t$ and $\overline{t}\equiv \overline{r} \pmod 6$. 
Clearly, $\overline{r}_1+\mathrm{c}_1=3$, $\overline{r}_2+\mathrm{c}_2=3$, $\overline{r}_3=\mathrm{c}_3$ and $\overline{r}_4=\mathrm{c}_4$. 
That is, all the conditions that Lemma \ref{6.2} required for this case are satisfied and then the solution orbit is parameterized by core abaci.
Consequently, if $r\in [(0, 0, 1, 1)]$, the corresponding solution orbit is parameterized by core by Lemma \ref{6.4}.

It is easy to check that 
\begin{align*}
  [(0,0,1,1)] = & \left\{(0,0,1,1),(0,0,1,5),(0,0,5,5)\right\} \\
  [(0,0,2,2)] = & \left\{(0,0,2,2),(0,0,2,4),(0,0,4,4)\right\} \\
  [(1,1,3,3)] = & \left\{(1,1,3,3),(1,3,3,5),(3,3,5,5)\right\} \\
  [(2,2,3,3)] = & \left\{(2,2,3,3),(2,3,3,4),(3,3,4,4)\right\} \\
  [(0,1,2,3)] = & \left\{(0,1,2,3),(0,1,3,4),(0,2,3,5),(0,3,4,5)\right\}
\end{align*}
in the sense of Definition $\ref{6.3}$. 
The proof is similar if $r$ is in the other equivalence classes and we omit the details here.
\end{proof}



\subsection{Open problems and related results} In this subsection, we shall propose two open problems about core abaci and give some relevant results.
Note that all the quadratic form content from number theory used in this section can be found in \cite{G, I}.

\smallskip

\noindent{\bf Open problem 1.} Given an affine type $X$, determine the set $\{|(\lam, j)|\mid (\lam, j)\in \mathcal{C}_j\}$.

\smallskip

Let us focus on the so-called positivity conjecture raised by Hanusa and Nath in \cite{HN} first. 
We will give an equivalent description for it.
\begin{conjecture}\label{6.7}
For affine type $C_3^{(1)}$, we have
$\{|\lam|\mid \lam\in \mathcal{C}_0\}=\mathbb{N}-\{2, 12, 13, 73\}$.
\end{conjecture}

\begin{remark}{\rm
In \cite{A} Alpoge proved the conjecture for large enough $n$ by applying deep results of Duke and Schulze-Pillot \cite{DS}.
However, the implied constant is ineffective because of the Landau–Siegel phenomenon. On the other hand, 
Hanson and Jameson \cite{HJ} proved the conjecture by assuming Generalized Riemann Hypothesis holds.}
\end{remark}

According to Subsections 5.1 and 5.2, we can give a new description of the set $\{|\lam|\mid \lam\in \mathcal{C}_0\}$.

\begin{proposition}\label{6.10}
For affine type $C_3^{(1)}$, we have
$$\{|\lam|\mid \lam\in \mathcal{C}_0\}= \{6(k_1^2+k_2^2+k_3^2)+3k_1+k_2+5k_3\mid k_1, k_2, k_3\in \mathbb{Z}\}.$$
\end{proposition}

\begin{proof}
Note that the equations of affine $C_3^{(1)}$ type with charge $0$ are of the form $x^2+y^2+z^2=24N+35$.
For a given nonnegative integer $N$, let $t=(t_1, t_2, t_3)$ be a solution of the equation.
It is easy to check that $t_1, t_2, t_3$ are all odd and ${\rm O}_{t}$ is in a unique $\sim_{12}$-equivalence class
$[(1, 1, 3)]$, $[(1, 3, 5)]$ or $[(3, 5, 5)]$ in the sense of Definition \ref{6.3}. By Lemmas \ref{6.2} and \ref{6.4}, all the solution orbits in 
the equivalence class $[(1, 3, 5)]$ can be parameterised by core abaci.

Without considering the order, the residue of a resolution in the equivalence class $[(1, 3, 5)]$ has to be one of the following eight cases.
$$\begin{array}{cccc}
    (1, 3, 5) & (1, 3, 7) & (1, 5, 9) & (1, 7, 9) \\
    (3, 5, 11) & (3, 7, 11) & (5, 9, 11) & (7, 9, 11).
  \end{array}$$
If the residue of $t=(t_1, t_2, t_3)$ is $(1, 3, 5)$, then assume that $t_1=12k_1+1$, $t_2=12k_2+3$ and $t_3=12k_3+5$.
By substituting them into the equation, we get $N=6(k_1^2+k_2^2+k_3^2)+3k_1+k_2+5k_3$.
If the residue of $t=(t_1, t_2, t_3)$ is $(1, 3, 7)$, assuming that $t_1=12k_1+1$, $t_2=12k_2+3$ and $t_3=12k_3+7$
yields $N=6(k_1^2+k_2^2+k_3^2)+3k_1+k_2+7k_3+1$. Let $\overline{k}_1=k_1$, $\overline{k}_2=k_2$ and $\overline{k}_3=-k_3-1$.
We obtain $N=6(\bar{k}_1^2+\bar{k}_2^2+\bar{k}_3^2)+3\overline{k}_1+\overline{k}_2+5\overline{k}_3$.
The other cases can be handled by similar method and we omit the details. As a result, we have proved ``$\subseteq$".

For the other direction ``$\supseteq$", fix $k_1, k_2, k_3\in \mathbb{Z}$ and denote $N=6(k_1^2+k_2^2+k_3^2)+3k_1+k_2+5k_3$.
We get $(12k_1+1)^2+(12k_2+3)^2+(12k_3+5)^2=24N+35$. It follows from Lemma \ref{6.1} that $(-12k_3-5, -12k_2-3, -12k_1-1)$
can be parameterized by core. This completes the proof.
\end{proof}

Based on Proposition \ref{6.10}, we have the following conjecture which is equivalent to Conjecture \ref{6.7}. 
We hope this equivalent description is useful to prove the positivity conjecture.

\begin{conjecture}\label{6.11}
Each number in $\mathbb{N}-\{2, 12, 13, 73\}$ is of the form $$6(k_1^2+k_2^2+k_3^2)+3k_1+k_2+5k_3,$$ where $k_1, k_2, k_3\in\mathbb{Z}$.
\end{conjecture}

In some cases, the set $\{|(\lam, j)|\mid (\lam, j)\in \mathcal{C}_j\}$ in {\bf Open problem 1} is just $\mathbb{N}$.
Note that in the frame of this paper, the role of $|\lam|$ is replaced by the corresponding height 
(which is equal to the atomic length, and a similar problem is considered in \cite{CGJL}). 
In non-branching cases, these two statistics are the same.
The following proposition provides some classes of examples.

\begin{proposition}\label{6.12}
For each nonnegative integer $N$, there exists a core abacus $(\lam, j)$ of affine type $X$ such that the height of $(\lam, j)$ is $N$,
where $X$ and $j$ are given in the following tableau.
\begin{table}[h] 
\centering
\begin{tabular}{|c|c|c|c|c|c|c|c|} 
\hline       		
{\rm type} & {\rm charge} & {\rm type} & {\rm charge} & {\rm type} & {\rm charge} & {\rm type} & {\rm charge} \\
\hline
$C_3^{(1)}$ & 1 & $D_4^{(2)}$ & 1 & $B_4^{(1)}$ & 2 &  $D_4^{(1)}$ & 2\\
\hline
\end{tabular}
\end{table}
\end{proposition}

\begin{proof}
We first prove the proposition for $1$-cores of affine type $C_3^{(1)}$ 
and the case of $1$-cores of the affine type $D_4^{(2)}$ is proved similarly.
It was already known to Gauss that a positive integer $n$ is representable 
as a sum of three squares if and only if $n\neq 4^k(8m-1)$.
Clearly, $24N+11$ is not the form of $4^k(8m-1)$, that is, 
the Diophantine equation $x^2+y^2+z^2=24N+11$ has integer solutions for arbitrary nonnegative $N$.
Then the result follows from Theorem \ref{6.6}. For affine types $B_4^{(1)}$ and $D_4^{(1)}$, 
we only need to point out that every positive integer can be 
represented as the sum of four squares.
\end{proof}

\smallskip

\noindent{\bf Open problem 2.} Enumerate core abaci of affine types with fixing certain parameters.

\smallskip

Closed formulae for enumerating certain cores will be given in some special cases. 
We first consider affine types $C_2^{(1)}$ and $D_3^{(2)}$, 
then $B_3^{(1)}$ and $D_4^{(2)}$, and finally $B_4^{(1)}$ and $D_4^{(1)}$.

The following lemma about the number of solutions parameterized by core abaci is important.
\begin{lemma}\label{6.13}
Given a Diophantine equation of affine type $X=C_2^{(1)}$ or $D_3^{(2)}$, then the number of solutions in a orbit  ${\rm O}_t$
being parameterized by core abaci is
$$|{\rm O}_t^{\rm p}|=
\begin{cases}
  1, & \mbox{\rm if } X=C_2^{(1)} \,{\rm or}\, X=D_3^{(2)}, j=0, 2; \\
  1, & \mbox{\rm if } X=C_2^{(1)} \,{\rm or}\, X=D_3^{(2)}, j=1 \,{\rm and}\, t=(t_1, t_2) \,{\rm with}\, |t_1| = |t_2|; \\
  2, & \mbox{\rm if } X=C_2^{(1)} \,{\rm or}\, X=D_3^{(2)}, j=1 \,{\rm and}\, t=(t_1, t_2) \,{\rm with}\, |t_1|\neq |t_2|.
\end{cases}
$$
\end{lemma}

\begin{proof}
It is a routine task of computation by Lemma \ref{6.4} and Corollary \ref{6.5}.
\end{proof}

We recall a notation from number theory that will appear in the closed formulae.
$$\chi_4(d)=
\begin{cases}
  1, & \mbox{if } d\equiv 1 \pmod 4; \\
  -1, & \mbox{if } d\equiv -1 \pmod 4;  \\
  0, & \mbox{if } d\equiv 0\,\, \mbox{or } 2 \pmod 4.
\end{cases}
$$

Denote the representation number of a positive integer $n$ represented as the sum of $k$ squares by $r_k(n)$.
Then we can calculate the number of core abaci of affine type $C_2^{(1)}$ and $D_3^{(2)}$ with arbitrary charge and a given height.
\begin{theorem}\label{6.14}
Denote the number of $j$-core abaci of affine type $X$ with height $N$ by $|\mathcal{C}_j^X(N)|$. Then 
$$|\mathcal{C}_j^X(N)|=
\begin{cases}
  \smallskip
  \frac{1}{8}r_2(16N+10)=\frac{1}{2}\sum\limits_{d\mid 16N+10}\chi_4(d), & \mbox{\rm if } j=0 \mbox{ \rm or } 2, \mbox{ \rm and } X=C_2^{(1)}; \\
  \smallskip
  \frac{1}{8}r_2(12N+5)=\frac{1}{2}\sum\limits_{d\mid 12N+5}\chi_4(d), & \mbox{\rm if } j=0 \mbox{ \rm or } 2, \mbox{ \rm and } X=D_3^{(2)}; \\
  \smallskip
  \frac{1}{4}r_2(16N+2)=\sum\limits_{d\mid 16N+2}\chi_4(d), & \mbox{\rm if } j=1 \mbox{ \rm and } X=C_2^{(1)}; \\
  \smallskip
  \frac{1}{4}r_2(6N+2)=\sum\limits_{d\mid 6N+2}\chi_4(d), & \mbox{\rm if } j=1 \mbox{ \rm and } X=D_3^{(2)}.
\end{cases}
$$
\end{theorem}

\begin{proof}
We first consider 0-cores of affine type $C_2^{(1)}$. It follows from Lemma \ref{6.13} that 
there is a bijection between the cores of height $N$ and solution orbits
for Diophantine equation $x^2+y^2=16N+10$, where $N$ is a nonnegative integer. Moreover, 
simple calculation yields that if $(t_1, t_2)$ is a solution of the Diophantine equations then $t_1\neq t_2$ and $t_i\neq 0$, $i=1,2$.
This implies that there are 8 solutions in each solution orbit of the Diophantine equations. 
We have from Dirichlet formula that $$r_2(16N+10)=4\sum\limits_{d\mid 16N+10}\chi_4(d)$$ and thus the proof of this case is complete.
Similarly, we can prove all of the other cases.
\end{proof}

Recall that a map $f: \mathbb{N}\longrightarrow \mathbb{N}$ 
is said to be multiplicative if $f(ab)=f(a)f(b)$ for ${\rm gcd}(a, b)=1$.
It is well-known that $\sum\limits_{d\mid n}\chi_4(d)$ is multiplicative. 
Then we can obtain a corollary of Proposition \ref{6.14}.

\begin{corollary}
Given two nonnegative integer $N_1, N_2$, then $|\mathcal{C}_1^{X}(N_1)||\mathcal{C}_1^{X}(N_2)|$ is equal to
$$
\begin{cases}
  |\mathcal{C}_1^{X}(8N_1N_2+N_1+N_2)|, & \mbox{\rm if } X=C_2^{(1)},\,  {\rm gcd}(8N_1+1, 8N_2+1)=1;\\
  |\mathcal{C}_1^{X}(3N_1N_2+N_1+N_2)|, & \mbox{\rm if } X=D_3^{(2)},\,  {\rm gcd}(3N_1+1, 3N_2+1)=1.
\end{cases}
$$
\end{corollary}

\begin{proof}
It is easy to check that $r_2(16N+2)=r_2(8N+1)$ by Dirichlet formula.
Since $\sum\limits_{d\mid n}\chi_4(d)$ is multiplicative, if ${\rm gcd}(8N_1+1, 8N_2+1)=1$, then
$$\frac{1}{4}r_2(8N_1+1)\cdot\frac{1}{4}r_2(8N_2+1)=\frac{1}{4}r_2(8(8N_1N_2+N_1+N_2)+1).$$
By Theorem \ref{6.14}, for 1-core abaci of affine type $C_2^{(1)}$, this implies that
$$|\mathcal{C}_1(N_1)||\mathcal{C}_1(N_2)|=|\mathcal{C}_1(8N_1N_2+N_1+N_2)|.$$
The proof of affine type $D_3^{(2)}$ is similar.
\end{proof}

Next let us study the affine types $B_3^{(1)}$ and $D_4^{(2)}$. 
Recall that $r_3(n)=\sum_{d^2|n}r_3^*(\frac{n}{d^2})$ (see \cite[Equation (4.8)]{G}), 
where $r_3^*$ is the number of primitive representations (Gauss expressed it in terms of the class number, 
for instance, see \cite[Equation (11.29)]{I} for the formula).

\begin{theorem}\label{6.16}
	Given a nonnegative integer $N$, we have
	$$|\mathcal{C}_j^X(N)|=
	\begin{cases}
		\frac{1}{24}r_3(8N+6),\, & \mbox{\rm if } X=D_4^{(2)}, j=2, \ N \,\ \text{\rm is even}; \\
		\frac{1}{48}r_3(8N+6),\,\, & \mbox{\rm if } X=D_4^{(2)}, j=2, \  N \,\ \text{\rm is odd}; \\
		\frac{1}{12}r_3(6N+2), & \mbox{\rm if } X=B_3^{(1)}, j=2,  \  N \,\ \text{\rm is odd}.
	\end{cases}
	$$
\end{theorem}

\begin{proof}
The equations of affine type $D_4^{(2)}$ with charge $2$ are of the form $x^2+y^2+z^2=8N+6$.
For a given nonnegative integer $N$, let $t=(t_1, t_2, t_3)$ be a solution of the equation.
Then two of $t_1, t_2, t_3$ are odd and ${\rm O}_{t}$ is in a unique $\sim_8$-equivalence class
$[(1, 1, 2)]$, $[(1, 2, 3)]$, $[(2, 3, 3)]$ in the sense of Definition \ref{6.3}.
By Corollary \ref{6.5}, we only need to take one solution orbit in each equivalence class to compute $|{\rm O}_t^{\rm p}|$.
If $t\in [(1, 1, 2)]$, without loss of generality, assume that $t_1=8k_1+1$, $t_2=8k_2+1$ and $t_3=8k_3+2$. 
According to the definition of a solution orbit and Lemma \ref{6.1}, one can obtain by routine calculation 
that $|{\rm O}_t|$ and the number of solutions in ${\rm O}_t$ parameterized by core abaci are as follows.
$$(|{\rm O}_t|, |{\rm O}_t^{\rm p}|)=
\begin{cases}
(24,\,1),  & {\rm if \,\,}  \text{ $|t_1|=|t_2|$}; \\
(48,\,2),  & {\rm if \,\,}  \text{ $|t_1|\neq|t_2|$}.
\end{cases}$$
	Similarly, if $t\in [(2, 3, 3)]$, we have
	$$(|{\rm O}_t|, |{\rm O}_t^{\rm p}|)=
	\begin{cases}
		(24,\,1),  & {\rm if \,\,}  \text{ $|t_1|=|t_2|$}; \\
		(48,\,2),  & {\rm if \,\,}  \text{ $|t_1|\neq|t_2|$}.
	\end{cases}$$
If $t\in [(1, 2, 3)]$, then $(|{\rm O}_t|, |{\rm O}_t^{\rm p}|)=(48, \,1)$.
	
If $N$ is even, then $8N+6\equiv 6 \pmod {16}$. This forces the squares of two odd numbers in $t$ simultaneously satisfying 
$x^2\equiv 1 \pmod{16}$, or $x^2\equiv 9 \pmod{16}$ by analysis of modulo 16. As a result, $t\in [(1, 1, 2)]$ or $[(2, 3, 3)]$.
This implies that $|{\rm O}_t|/ |{\rm O}_t^{\rm p}|=24$ and thus $|\mathcal{C}_j^X(N)|=\frac{1}{24}r_3(8N+6)$.
On the other hand, for $N$ being odd, the squares of two odd numbers in $t$ satisfy that one is of the form $16n_1+1$ and the other one $16n_2+9$.
That is, $t\in [(1, 2, 3)]$ and hence $|{\rm O}_t|/ |{\rm O}_t^{\rm p}|=48$. The proof of affine type $D_4^{(2)}$ is complete.
	
For affine type $B_3^{(1)}$, the corresponding equations are $x^2+y^2+z^2=6N+2$. 
If $N$ is odd, then $4|6N+2$. Based on this fact, we can obtain by direct calculation 
that a solution $t$ of the equation is in the $\sim_6$ equivalence class $[(0, 2, 2)]$.
Without loss of generality, assume that $t_1=6k_1$, $t_2=6k_2+2$ and $t_3=6k_3+2$. Then
$$(|{\rm O}_t|, |{\rm O}_t^{\rm p}|)=
\begin{cases}
(12,\,1),  &  {\rm if }\,\, \text{ $k_1=0$, $|t_2|=|t_3|$}; \\
(24,\,2),  & {\rm if }\,\, \text{ $k_1=0$, $|t_2|\neq|t_3|$};\\
(24,\,2),  & {\rm if }\,\, \text{ $k_1\neq0$, $|t_2|=|t_3|$};\\
(48,\,4),  & {\rm if }\,\, \text{ $k_1\neq0$, $|t_2|\neq|t_3|$.}
\end{cases}$$
That is, $|{\rm O}_t|/|{\rm O}_t^{\rm p}|=12$ and the proof is complete.
\end{proof}

The number $p_l(n)$ of partitions of $n$ into $l$ squares is clearly meaningful. Unfortunately, as far as we know, there is not a closed formula for it in general.
According to the results obtained above, we can partially deal with this problem for small ranks. 
The following corollary is obtained from the proof of Theorems \ref{6.14} and \ref{6.16}.

\begin{corollary}
Keep notations as above and let $N$ be a nonnegative integer. We have
\begin{enumerate}
  \item[\rm (1)]\, $p_2(16N+10)=\frac{1}{2}\sum\limits_{d\mid 16N+10}\chi_4(d)$;
  \item[\rm (2)]\, $p_2(12N+5)= \frac{1}{2}\sum\limits_{d\mid 12N+5}\chi_4(d)$;
  \item[\rm (3)]\, $p_3(8N+6)= \frac{1}{48}r_3(8N+6)$ if $N$ is odd.
\end{enumerate}
\end{corollary}

Recall that Jacobi gave the formula $r_4(n)=8(2+(-1)^n)\sum\limits_{d|n, 2\nmid d}d$.
Now we can consider affine types $B_4^{(1)}$ and $D_4^{(1)}$. 

\begin{theorem}
Given a nonnegative integer $N$, we have
$$|\mathcal{C}_j^X(N)|=
\begin{cases}
  \frac{1}{96}r_4(8N+6)=\frac{1}{4}\sum\limits_{d\mid 8N+6, 2\nmid d}d,\, & \mbox{\rm if } X=B_4^{(1)}, j=2, \ N \,\ \text{\rm is even}; \\
  \frac{1}{192}r_4(8N+6)=\frac{1}{8}\sum\limits_{d\mid 8N+6, 2\nmid d}d,\,\, & \mbox{\rm if } X=B_4^{(1)}, j=2, \  N \,\ \text{\rm is odd}; \\
  \frac{1}{24}r_4(6N+2)=\sum\limits_{d\mid 6N+2, 2\nmid d}d, & \mbox{\rm if } X=D_4^{(1)}, j=2,  \  N \,\ \text{\rm is odd}.
\end{cases}
$$
\end{theorem}

\begin{proof}
Let us consider affine type $B_4^{(1)}$ first. Note that the equations 
of affine $B_4^{(1)}$ type with charge $2$ are of the form $x_1^2+x_2^2+x_3^2+x_4^2=8N+6$.
For a given nonnegative integer $N$, let $t=(t_1, t_2, t_3, t_4)$ be a solution of the equation.
It is easy to check that two of $t_1, t_2, t_3, t_4$ are odd and ${\rm O}_{t}$ is in a unique $\sim_8$-equivalence class
$[(0, 1, 1, 2)]$, $[(0, 1, 2, 3)]$, $[(0, 2, 3, 3)]$, $[(1, 1, 2, 4)]$, $[(1, 2, 3, 4)]$ or $[(2, 3, 3, 4)]$ in the sense of Definition \ref{6.3}.
By Corollary \ref{6.5}, we only need to take one solution orbit in each equivalence class to compute $|{\rm O}_t^{\rm p}|$.
If $t\in [(0, 1, 1, 2)]$, without loss of generality, assume that $t_1=8k_1$, $t_2=8k_2+1$, $t_3=8k_3+1$ and $t_4=8k_4+2$. 
Then according to the definition of a solution orbit and Lemma \ref{6.1}, one can get by routine calculation 
that $|{\rm O}_t|$ and the number of solutions in ${\rm O}_t$ parameterized by core abaci are as follows.
$$(|{\rm O}_t|, |{\rm O}_t^{\rm p}|)=
\begin{cases}
		(96,\,1),  & {\rm if \,\,}  \text{ $k_1=0$, $|t_1|=|t_2|$}; \\
		(192,\,2),  & {\rm if \,\,}  \text{ $k_1\neq 0$, $|t_1|=|t_2|$};\\
		(192,\,2),  & {\rm if \,\,}  \text{ $k_1= 0$, $|t_1|\neq|t_2|$};\\
		(384,\,4),  & {\rm if \,\,}  \text{ $k_1\neq0$, $|t_1|\neq|t_2|$}.
\end{cases}$$
Similarly, if $t\in [(0, 1, 2, 3)]$, the corresponding datum is
$$(|{\rm O}_t|, |{\rm O}_t^{\rm p}|)=
\begin{cases}
	(192,\,1),  &  {\rm if \,\,} \text{ $k_1=0$}; \\
	(384,\,2),  &  {\rm if \,\,} \text{ $k_1\neq 0$.}
\end{cases}$$
If $t\in [(0, 2, 3, 3)]$, we have
$$(|{\rm O}_t|, |{\rm O}_t^{\rm p}|)=
\begin{cases}
		(96,\,1),  & {\rm if \,\,}  \text{ $k_1=0$, $|t_3|=|t_4|$}; \\
		(192,\,2),  & {\rm if \,\,}  \text{ $k_1\neq 0$, $|t_3|=|t_4|$};\\
		(192,\,2),  & {\rm if \,\,}  \text{ $k_1= 0$, $|t_3|\neq|t_4|$};\\
		(384,\,4),  & {\rm if \,\,}  \text{ $k_1\neq0$, $|t_3|\neq|t_4|$}.
\end{cases}$$
If $t\in [(1, 1, 2, 4)]$, then
$$(|{\rm O}_t|, |{\rm O}_t^{\rm p}|)=
\begin{cases}
		(192,\,2),  & {\rm if \,\,}  \text{ $|t_1|=|t_2|$};\\
		(384,\,4),  & {\rm if \,\,}  \text{ $|t_1|\neq|t_2|$}.
\end{cases}$$
If $t\in [(2, 3, 3, 4)]$, then
$$(|{\rm O}_t|, |{\rm O}_t^{\rm p}|)=
\begin{cases}
		(192,\,2),  & {\rm if \,\,}  \text{ $|t_2|=|t_3|$};\\
		(384,\,4),  & {\rm if \,\,}  \text{ $|t_2|\neq|t_3|$}.
\end{cases}$$
Finally, if $t\in [(1, 2, 3, 4)]$, then $(|{\rm O}_t|, |{\rm O}_t^{\rm p}|)=(384, \,2)$.

If $N$ is even, then $8N+6\equiv 6 \pmod {16}$. This forces the squares of two odd numbers in $t$ simultaneously satisfying 
$x^2\equiv 1 \pmod{16}$, or $x^2\equiv 9 \pmod{16}$ by analysis of modulo 16. 
As a result, $t\in [(0, 1, 1, 2)]$, $[(0, 2, 3, 3)]$, $[(1, 1, 2, 4)]$ or $[(2, 3, 3, 4)]$.
This implies that $|{\rm O}_t|/ |{\rm O}_t^{\rm p}|=96$ and thus $|\mathcal{C}_j^X(N)|=\frac{1}{96}r_4(8N+6)$.
On the other hand, for $N$ being odd, the squares of two odd numbers in $t$ 
satisfy that one is of the form $16n_1+1$ and the other one $16n_2+9$.
That is, $t\in [(0, 1, 2, 3)]$ or $[(1, 2, 3, 4)]$ and hence $|{\rm O}_t|/ |{\rm O}_t^{\rm p}|=192$. 
The proof of affine type $B_4^{(1)}$ is complete.

We now consider affine type $D_4^{(1)}$. Note that the corresponding equations are of the form $x_1^2+x_2^2+x_3^2+x_4^2=6N+2$. 
If $N$ is odd, then $4|6N+2$. Based on this fact, we can obtain by direct calculation 
that a solution $t$ of the equation is in a $\sim_6$ equivalence class $[(0, 0, 2, 2)]$ or $[(1, 1, 3, 3)]$.
If $t\in [(0, 0, 2, 2)]$, without loss of generality, assume that $t_1=6k_1$, $t_2=6k_2$, $t_3=6k_3+2$ and $t_4=6k_4+2$. Then
$$(|{\rm O}_t|, |{\rm O}_t^{\rm p}|)=
\begin{cases}
	(24,\,1),  &  {\rm if }\,\, \text{ $k_1=k_2=0$, $|t_3|=|t_4|$}; \\
	(48,\,2),  & {\rm if }\,\, \text{ $k_1=k_2=0$, $|t_3|\neq|t_4|$};\\
	(96,\,4),  & {\rm if }\,\, \text{ $k_1k_2=0$, $k_1\neq k_2$, $|t_3|=|t_4|$};\\
    (192,\,8),  & {\rm if }\,\, \text{ $k_1k_2=0$, $k_1\neq k_2$, $|t_3|\neq|t_4|$};\\
	(96,\,4),  & {\rm if }\,\, \text{ $k_1k_2\neq0$, $|t_1|=|t_2|$, $|t_3|=|t_4|$}; \\
	(192,\,8),  & {\rm if }\,\, \text{ $k_1k_2\neq0$, $|t_1|=|t_2|$, $|t_3|\neq|t_4|$};\\
	(192,\,8),  & {\rm if }\,\, \text{ $k_1k_2\neq0$, $|t_1|\neq|t_2|$, $|t_3|=|t_4|$};\\
	(384,\,16),  & {\rm if }\,\, \text{ $k_1k_2\neq0$, $|t_1|\neq|t_2|$, $|t_3|\neq|t_4|$.}
\end{cases}
$$
If $t\in [(1, 1, 3, 3)]$, then
$$(|{\rm O}_t|, |{\rm O}_t^{\rm p}|)=
\begin{cases}
	(96,\,4),  & {\rm if }\,\, \text{ $|t_1|=|t_2|$, $|t_3|=|t_4|$}; \\
	(192,\,8),  & {\rm if }\,\, \text{ $|t_1|=|t_2|$, $|t_3|\neq|t_4|$};\\
	(192,\,8),  & {\rm if }\,\, \text{ $|t_1|\neq|t_2|$, $|t_3|=|t_4|$};\\
	(384,\,16),  & {\rm if }\,\, \text{ $|t_1|\neq|t_2|$, $|t_3|\neq|t_4|$.}
\end{cases}
$$
That is, $|{\rm O}_t|/|{\rm O}_t^{\rm p}|=24$ for all cases and the proof is complete.
\end{proof}

\bigskip

\bigskip

\noindent{\bf Acknowledgement}
The work is partially supported by the Open Project Program of Key Laboratory of Mathematics and Complex
System, Beijing Normal University (Grant No. K202402).
The authors would like to thank Dr. Xiangyu Qi and Feiyue Huang for some helpful conversations.
Part of this work was done when Li visited Key Laboratory of Mathematics and Complex
System at Beijing Normal University in July 2025, and when Li visited Department of Mathematics of National University of Singapore in Jan. 2026. 
He takes this opportunity to express his sincere thanks to
Prof. W. Hu  and Prof. K. M. Tan for the hospitality during the visits. 



\bigskip\bigskip


\begin{thebibliography}{}

\bibitem{A} L. Alpoge, {\em Self-conjugate core partitions and modular forms}, J. Number Theory, {\bf 140} (2014) 60-92.

\bibitem{BDF} J. Baldwin, M. Depweg, B. Ford, A. Kunin and L. Sze, {\em Self-conjugate t-core partitions, sums
of squares, and p-blocks of An}, J. Algebra, {\bf 297} (2006) 438-452.

\bibitem{BB} A. Björner and F. Brenti, Combinatorics of Coxeter groups, Graduate Texts in Mathematics, Vol. 231 (Springer,
New York, 2005). Page 39, 40

\bibitem{BCG} O. Brunat, N. Chapelier-Laget and T. Gerber, {\em Generalised core partitions and Diophantine equations}, arXiv: 2403.11191.

\bibitem{CG} N. Chapelier-Laget and T. Gerber, {\em Atomic length on Weyl groups}, J. Comb. Algebra (2024), published online first.

\bibitem{CGJL} N. Chapelier-Laget, T. Gerber, N. Jacon and C. Lecouvey, {\em Entropy of affine permutations and universality of affine
atomic lengths}, arXiv: 2603.22256.

\bibitem{DS} W. Duke and R. Schulze-Pillot, {\em Representations of integers by positive ternary quadratic forms and
equidistribution of lattice points on ellipsoids}, Invent. Math., {\bf 99} (1990) 49-57.


\bibitem{F2} M. Fayers, {\em Core blocks of Ariki-Koike algebras}, J. Algebr. Comb., {\bf 26} (2007) 47-81.

\bibitem{GKS} F. G. Garvan, D. Kim, and D. Stanton, {\em Cranks and t-cores}, Invent. Math., {\bf 101} (1990) 1-17.

\bibitem{GO} A. Granville and K. Ono, {\em Defect zero p-blocks for finite simple groups}, Trans. AMS, {\bf 348} (1996) 331-347.

\bibitem{G} E. Grosswald, Representations of integers as sums of squares,  Springer-Verlag New York Inc. (1985).

\bibitem{HJ} M. Hanson and M. Jameson, {\em Self-conjugate 6-cores and quadratic forms}, arXiv: 2211.00738.

\bibitem{HJ2} C. Hanusa and B. Jones, {\em Abacus models for parabolic quotients of affine Weyl groups}, J. Algebra, {\bf 361} (2012) 134-162.

\bibitem{HN} C. Hanusa and R. Nath, {\em The number of self-conjugate core partitions}, J. Number Theory, {\bf 133} (2013) 751-768.

\bibitem{HLQ} W. Hu, F. Huang, Y. Li and X. Qi, {\em Moving vectors and level-rank duality}, arXiv: 2604.25340.

\bibitem{H}  J.E. Humphreys, Reflection Groups and Coxeter Groups, Cambridge University Press; 1990.

\bibitem{I} H. Iwaniec, Topics in classical automorphic forms, Graduate Studies in Mathematics, Vol. {\bf 17} the American Mathematical Society 1997.

\bibitem{JL} N. Jacon and C. Lecouvey, {\em Cores of Ariki-Koike algebras}, Doc. Math., {\bf 26} (2021) 103-124.

\bibitem{J} G. James, {\em Some combinatorial results involving Young diagrams}, Proc. Cambridge Philos. Soc., {\bf 83} (1978) 1-10.

\bibitem{Kac} V.~G. Kac, Infinite dimensional Lie algebras, CUP, Cambridge, third~ed., 1994.

\bibitem{LW} C. Lecouvey and D. Wahiche, {\em Macdonald identities, Weyl–Kac denominator
formulas and affine Grassmannian elements}, SIGMA, {\bf 21} (2025) 023, 45 pages.

\bibitem{LQ} Y. Li and X. Qi, {\em Moving vectors I: Representation type of blocks of Ariki-Koike algebras},
 J. London Math. Soc. (2), {\bf 111} (2025): e70169.

\bibitem{LQT} Y. Li, X. Qi and K. M. Tan, {\em Moving vectors and core blocks of Ariki-Koike algebras},
 J. Algebra, {\bf 694} (2026) 497-563.
 



\bibitem{M} I. G. Macdonald, Symmetric functions and Hall polynomials, Oxford Mathematical Monographs, second edition, 1998.

\bibitem{NO} N. Nekrasov and A. Okounkov, Seiberg-Witten theory and random partitions, The unity of mathematics, Progr. Math., vol. 244,
Birkhäuser Boston, Boston, MA, 2006, pp. 525–596.

\bibitem{P} M. Pétréolle, Quelques développements combinatoires autour des groupes de Coxeter et des partitions d’entiers, Theses, Université
Claude Bernard-Lyon I, November 2015.

\bibitem{STW} E. N. Stucky, M. Thiel and N. Williams, {\em Strange expectations in affine Weyl
groups}, arXiv:2309.14481.

\bibitem{TW} M. Thiel and N. Williams, {\em Strange expectations and simultaneous cores}, J. Algebr. Comb., {\bf 46} (2017) 219-261.

\bibitem{U} D. Uglov, {\em Canonical bases of higher level q-deformed Fock spaces and Kazhdan-Lusztig polynomials}, in Physial
Combinatorics (ed. M. Kashiwara, T. Miwa), Progress in Math. 191, Birkhauser (2000).

\bibitem{W} D. Wahiche, {\em Some combinatorial interpretations of the Macdonald identities for affine root systems and Nekrasov-Okounkov type
formulas}, arXiv 2306.08071.

\end{thebibliography}
\end{document}